\documentclass[11pt]{amsart}

\usepackage{amsmath,amssymb,amsthm,accents}
\usepackage{amscd}
\usepackage[initials,short-journals,non-sorted-cites]{amsrefs}
\usepackage{enumitem}
\usepackage{type1cm}
\usepackage{float}
\usepackage[all]{xy}
\usepackage{amsfonts}
\usepackage[dvipdfmx]{graphicx}
\usepackage{comment}
\usepackage{bm}

\usepackage{color}
\allowdisplaybreaks[4]

\renewcommand{\abstract}[1]{
	\begin{center}
	\parbox{13cm}{\small {\sc Abstract.} #1}
	\end{center}
	\smallskip
}

\newtheorem{theo}{Theorem}[section]
\newtheorem{prop}[theo]{Proposition}
\newtheorem{lemm}[theo]{Lemma}
\newtheorem{cor}[theo]{Corollary}

\theoremstyle{definition}

\newtheorem{exam}[theo]{Example}

\newtheorem{remark}[theo]{Remark}

\newcommand{\Z}{\mathbb{Z}}

\newcommand{\R}{\mathbb{R}}

\title[Twisted Alexander matrices of quandles]{Twisted Alexander matrices of quandles associated with a certain Alexander pair}
\author{Yuta Taniguchi}
\address{\scriptsize DEPARTMENT OF MATHEMATICS, GRADUATE SCHOOL OF SCIENCE, OSAKA UNIVERSITY, 1-1, MACHIKANEYAMA, TOYONAKA, OSAKA, 560-0043, JAPAN}
\email{yuta.taniguchi.math@gmail.com}

\begin{document}
% 57M25 Knots and links in $S^3$ 
% 57M27 Invariants of knots and 3-manifolds
% 57K10 Knot theory 
% 57K12 Generalized knots (virtual knots, welded knots, quandles, etc.)

\maketitle

\abstract{Ishii and Oshiro introduced the notion of an $f$-twisted Alexander matrix, which is a quandle version of a twisted Alexander matrix and defined an invariant of finitely presented quandles. In this paper, we study $f$-twisted Alexander matrices of certain quandles with the Alexander pair obtained from a quandle $2$-cocycle. We show that the 0-th elementary ideal of  $f$-twisted Alexander matrix of the knot quandle of a surface knot with the Alexander pair obtained from a quandle $2$-cocycle can be described with the Carter-Saito-Satoh's invariant. We also discuss a relationship between $f$-twisted Alexander matrices of connected quandles with the Alexander pair obtained from a quandle $2$-cocycle and quandle homology groups.}

\section{Introduction}
The Alexander polynomial \cite{Alexander1928topo} is one of the most important invariant of a knot. We can compute the Alexander polynomial from the Alexander matrix. Using Fox free derivatives \cite{Fox1953free}, the Alexander matrix is obtained from a presentation of the knot group. In 1990's, X. S. Lin \cite{Lin2001repr} defined the twisted Alexander polynomial associated with a linear representation and M. Wada \cite{Wada1994twisted} generalized this notion for finitely presented group using Fox free derivatives. 

A quandle \cite{Joyce1982quandle, Matveev1982distributive} is a set with a binary operation whose axioms correspond to the Reidemeister moves. Given an oriented any dimensional knot, we have a quandle associated with the knot, which is called the knot quandle. It is known that the knot quandle is more useful than the knot group for distinguishing knots.  In \cite{Ishii2022twisted}, A. Ishii and K. Oshiro introduced derivatives for quandles and the notion of an $f$-twisted Alexander matrix. Using an $f$-twisted Alexander matrix, we can obtain an invariant for finitely presented quandles. Futhermore, they showed that the (twisted) Alexander polynomial of a knot can be recovered from an $f$-twisted Alexander matrix by setting well (see also \cite{Ishii2022quandle}). To obtain an $f$-twisted Alexander matrix, we need to fix an Alexander pair, which is a pair of maps satisfy certain conditions. In \cite{Taniguchitwisted}, the author pointed out that an Alexander pair can be obtained from a quandle 2-cocycle, and showed that the 0-th elementary ideal of the $f$-twisted Alexander matrix of the knot quandle of a classical knot with the Alexander pair obtained from a quandle $2$-cocycle is determined by the quandle cocycle invariant \cite{Carter2003quandle}. Furthermore, using $f$-twisted Alexander matrices with a certain Alexander pair, he distinguished two classical knots which can not be distinguished by the (twisted) Alexander polynomial.

The purpose of this paper is to study $f$-twisted Alexander matrices of certain quandles with the Alexander pair obtained from a quandle $2$-cocycle. First, we focus on $f$-twisted Alexander matrices of knot quandles of surface knots with the Alexander pair obtained from a quandle 2-cocycle. We show that the 0-th elementary ideal of the $f$-twisted Alexander matrix of the knot quandle of a surface knot with the Alexander pair obtained from a quandle 2-cocycle is determined by the Carter-Saito-Satoh's invariant \cite{Carter2006ribbon} (Theorem \ref{theo:main_1}). By the definition, the Carter-Saito-Satoh's invariant of a 2-knot, which is a 2-sphere embedded in $\R^4$, is trivial. This implies that the 0-th elementary ideal of the $f$-twisted Alexander matrix of the knot quandle of a 2-knot with the Alexander pair obtained from a quandle 2-cocycle is trivial (Corollary \ref{cor:elementary_ideal_surface_knot}). Second, we discuss $f$-twisted Alexander matrices of connected quandles with the Alexander pair obtained from a quandle $2$-cocycle. We prove that the 0-th elementary ideal of the $f$-twisted Alexander matrix of a connected quandle with the Alexander pair obtained from a quandle 2-cocycle can be realized by the second quandle homology group of the quandle (Theorem \ref{theo:main_2}). Combining our results, we see that the second quandle homology group of the knot quandle of a 2-knot is trivial (Corollary \ref{cor:2-knot_homology}).

This paper is organized as follows: In section \ref{sect:qdle_presentation}, we recall the definition of a quandle and presentations of quandle. In section \ref{sect:def_twisted_Alexander}, we review the definition of an $f$-twisted Alexander matrix. In section \ref{sect:Alexander_matrix_surface_knot}, we discuss $f$-twisted Alexander matrices of knot quandles of surface knots using the Alexander pair associated with a quandle 2-cocycle. In section \ref{sect:Alexander_matrix_connected_qdle}, we study $f$-twisted Alexander matrices of connected quandles using the Alexander pair associated with a quandle 2-cocycle and determine the second quandle homology group of the knot quandle of 2-knots.

\section*{Acknowledgements}
The author would like to thank Seiichi Kamada for many invaluable advice. This work was supported by JSPS KAKENHI Grant Number 21J21482.

\section{Quandles and quandle presentations}
\label{sect:qdle_presentation}
A {\it quandle}~\cite{Joyce1982quandle, Matveev1982distributive} is a non-empty set $X$ with a binary operation $X^2\to X;(x,y)\mapsto x^y$ satisfying the following axioms:\vspace{-2mm}
\begin{itemize}
	\setlength{\itemsep}{0pt}
	\setlength{\parskip}{0pt}
\item[(Q1)] For any $x\in X$, we have $x^x=x$.
\item[(Q2)] For any $x,y\in X$, there exists unique element $z$ such that $z^y=x$.
\item[(Q3)] For any $x,y,z\in X$, we have $(x^y)^z=(x^z)^{(y^z)}$.
\end{itemize}
We denote the element $z$ appeared in the axiom (Q2) by $x^{y^{-1}}$. In this paper, $(x^y)^z$ is denoted by $x^{yz}$.

Let $X$ and $Y$ be quandles. A map $f:X\to Y$ is a {\it quandle homomorphism} if $f(x^y)=f(x)^{f(y)}$ for any $x,y\in X$. A quandle homomorphism $f:X\to Y$ is a {\it quandle isomorphism} if $f$ is a bijection. We denote by ${\rm Hom}(X,Y)$ the set of all homomorphism from $X$ to $Y$.

\begin{exam}
Let $\mathcal{K}$ be an oriented, connected, closed $n$-dimensional submanifold in $\R^{n+2}$. Let $Q(\mathcal{K})$ be the set of all homotopy classes, $x=[(D,\alpha)]$, of all pairs $(D,\alpha)$, where $D$ is a meridian disk of $\mathcal{K}$ and $\alpha$ is a path from a point in $\partial D$ and ending at a fixed base point $\ast\in \R^{n+2}\backslash\mathcal{K}$. Then, $Q(\mathcal{K})$ is a quandle with an operation defined by
\[
[(D_1,\alpha)]^{[(D_2,\beta)]}=[(D_1,\alpha\cdot\beta^{-1}\cdot\partial D_2\cdot \beta)],
\]
where $\partial D_2$ is a meridian loop starting from the initial point of $\beta$ and going along $\partial D_2$ in the positive direction. We call this quandle $Q(\mathcal{K})$ the {\it knot quandle} of $\mathcal{K}$.
\end{exam}
Next, we review the definition of presentations of quandles. Let $S$ be a non-empty set. The {\it free quandle} on $S$ denoted by $FQ(S)$ is a quandle with a map $\mu:S\to FQ(S)$ such that for any quandle $X$ and any map $f:S\to X$, there exists unique quandle homomorphism $f_{\#}:FQ(S)\to X$ such that $f=f_{\#}\circ\mu$. 

For $R\subset FQ(S)^2$, let us define the congruence $\sim_R$ on $FQ(S)$ to be the smallest congruence containing $R$.  Then, the set $\langle S\mid R\rangle:=FQ(S)/\sim_R$ has the quandle operation inherited from $FQ(S)$. We call the elements of $S$ the {\it generators} of $\langle S\mid R\rangle$ and the elements of $R$ the {\it relators} of $\langle S\mid R\rangle$. If a quandle $X$ is isomorphic to $\langle S\mid R\rangle$, we say that $X$ has a presentation $\langle S\mid R\rangle$. A presentation $\langle S\mid R\rangle$ is {\it finite} if $S$ and $R$ are finite sets. Refer to \cite{Fenn1992racks, Kamada2017surface} for details.

R. Fenn and C. Rouke showed \cite{Fenn1992racks} that $\langle S_1\mid R_1\rangle$ and $\langle S_2\mid R_2\rangle$ are isomorphic if and only if these presentations are related by a finite sequence of the following operations:
\begin{enumerate}
\setlength{\leftskip}{0.5cm}
\setlength{\itemsep}{0pt}
\setlength{\parskip}{0pt}
\item[(T1-1)] $\langle S\mid R\rangle\leftrightarrow\langle S\mid R\cup\{(x,x)\}\rangle$ $(x\in FQ(S))$.
\item[(T1-2)] $\langle S\mid R\cup\{(x,y)\}\rangle\leftrightarrow\langle S\mid R\cup\{(x,y), (y,x)\}\rangle$.
\item[(T1-3)] $\langle S\mid R\cup\{(x,y), (y,z)\}\rangle\leftrightarrow\langle S\mid R\cup\{(x,y), (y,z),(x,z)\}\rangle$.
\item[(T1-4)] $\langle S\mid R\cup\{(x,y)\}\rangle\leftrightarrow\langle S\mid R\cup\{(x,y),(x^{z^{\varepsilon}},y^{z^{\varepsilon}})\}\rangle$ $(z\in S,\varepsilon\in\{\pm 1\})$.
\item[(T1-5)] $\langle S\mid R\cup\{(x,y)\}\rangle\leftrightarrow\langle S\mid R\cup\{(x,y),(z^x,z^y)\}\rangle$ $(z\in FQ(S))$.
\item[(T2)] $\langle S\mid R\rangle\leftrightarrow\langle S\cup\{y\}\mid R\cup\{(y,w_y)\}\rangle$ $(y\notin S,w_y\in FQ(S))$.
\end{enumerate}
These operations are called {\it Tietze's moves}.
\begin{exam}
{\rm
Let $D$ be a diagram of an oriented classical knot $K$ in $\R^3$ and ${\rm Arc}(D)$ the set of arcs of $D$. For each crossing $\chi$, we denote a relator $r_\chi$ by $(x_i^{x_j},x_k)$, where $x_i,x_j,x_k$ are the arcs around $\chi$ such that the normal orientation of the over arc $x_j$ points from $x_i$ to $x_k$ (see Figure \ref{crossing_condition}). 
\begin{figure}[h]
	 \centering{\includegraphics[height=3cm]{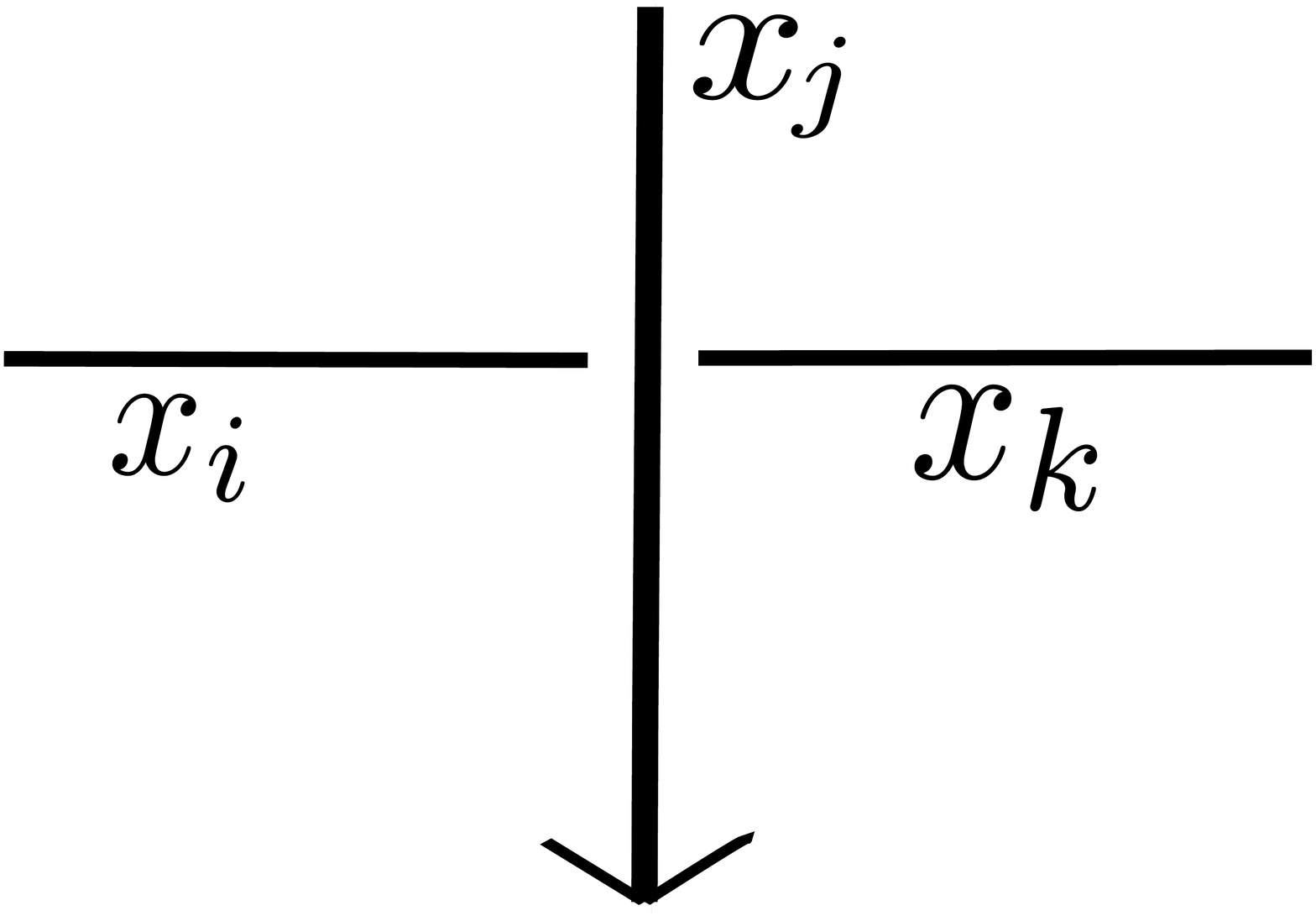}}
	\caption{The arcs around the crossing $\chi$}
	\label{crossing_condition}
 \end{figure}
Then, the knot quandle $Q(K)$ has the following presentation:
\[
\langle {\rm Arc}(D)\mid \{ r_\chi\mid \chi:\textrm{ a crossing of }D\}.
\]
}
\end{exam}

\section{$f$-twisted Alexander matrices}
\label{sect:def_twisted_Alexander}
Let $X$ be a quandle and $R$ a ring with the unity $1$. A pair $(f_1,f_2)$ of maps $f_1,f_2:X^2\to R$ is an {\it Alexander pair} \cite{Ishii2022twisted} if $f_1$ and $f_2$ satisfy the following axioms:
 \begin{itemize}
	\setlength{\itemsep}{0pt}
	\setlength{\parskip}{0pt}
\item For any $x\in X$, we have $f_1(x,x)+f_2(x,x)=1$.
\item For any $x,y\in X$, $f_1(x,y)$ is a unit of $R$.
\item For any $x,y,z\in X$, we have 
\begin{align*}
&f_1(x^y,z)f_1(x,y)=f_1(x^z,y^z)f_1(x,z),\\
&f_1(x^y,z)f_2(x,y)=f_2(x^z,y^z)f_1(y,z), {\rm and}\\
&f_2(x^y,z)=f_1(x^z,y^z)f_2(x,z)+f_2(x^z,y^z)f_2(y,z).
\end{align*}
\end{itemize}
In \cite{Andruskiewitsch2003racks},  N. Andruskiewitsch and M. Gra\~{n}a studied the extension of quandles and introduced a {\it dynamical cocycle}. An Alexander pair is a dynamical cocycle corresponding to a {\it quandle module}.

\begin{exam}
\label{exam:Alexander_pair_1}
Let $X$ be a quandle and $\Z[t^{\pm 1}]$ the ring of Laurent polynomials with integer coefficients. Let us define maps $f_1,f_2:X^2\to \Z[t^{\pm 1}]$ by
\[
f_1(x,y):=t,\quad f_2(x,y):=1-t.
\]
Then, the pair $(f_1,f_2)$ is an Alexander pair.
\end{exam}

\begin{exam}
\label{exam:Alexander_pair_2}
Let $X$ be a quandle and $A$ an abelian group. A map $\theta:X^2\to A$ is a {\it quandle $2$-cocycle} \cite{Carter2003quandle} if $\theta$ satisfies the following conditions:
\begin{itemize}
\setlength{\itemsep}{0pt}
\setlength{\parskip}{0pt}
\item For any $x\in X$, we have $\theta(x,x)=0_A$, where $0_A$ is the identity element.
\item For any $x,y,z\in X$, we have $\theta(x,y)+\theta(x^y,z)=\theta(x,z)+\theta(x^z,y^z)$. 
\end{itemize}
Let $\Z[A]$ be the group ring. For a quandle $2$-cocycle $\theta:X^2\to A$, we define the map $f_{\theta},0:X^2\to\Z[A]$ by
\[
f_{\theta}(x,y):=1\cdot\theta(x,y),\quad 0(x,y):=0.
\]
By the direct calculation, we see that $f_{\theta}=(f_{\theta},0)$ is an Alexander pair, which is called the {\it Alexander pair associated with the quandle $2$-cocycle} \cite{Taniguchitwisted}. 
\end{exam}

Next, we will review the definition of $f$-twisted Alexander matrices introduced in \cite{Ishii2022twisted}. Refer to \cite{Ishii2022twisted} for details. Let $S=\{ x_1,\ldots,x_n\}$ be a finite set and $Q$ a quandle with a finite presentation $\langle x_1,\ldots,x_n\mid r_1,\ldots,r_m\rangle$. Let ${\rm pr}_{\rm qdle}:FQ(S)\to Q$ be the canonical projection. In this paper, we often omit pr to represent ${\rm pr}_{\rm qdle}(a)$ as $a$. 

Let $R$ be a ring and $f=(f_1,f_2)$ an Alexander pair of maps $f_1,f_2:Q\times Q\to R$. For $j\in\{1,\ldots,n\}$, let us define a map $\frac{\partial_{f}}{\partial{x_j}}:FQ(S)\to R$ by the following rules:
\begin{itemize}
\item For any $x,y\in FQ(S)$, we have $\frac{\partial_{f}}{\partial{x_j}}(x^y)=f_1(x,y)\frac{\partial_{f}}{\partial{x_j}}(x)+f_2(x,y)\frac{\partial_{f}}{\partial{x_j}}(y)$.
\item For each $i\in\{1,\ldots,n\}$, we have $\frac{\partial_{f}}{\partial{x_j}}(x_i)=\begin{cases}
1\ (i=j)\\
0\ (i\neq j)
\end{cases}$.
\end{itemize}

The map $\frac{\partial_{f}}{\partial{x_j}}:FQ(S)\to R$ is called the {\it $f$-derivative with respect to $x_j$} \cite{Ishii2022twisted}. Using the first rule, we see that 
\[
\frac{\partial_{f}}{\partial{x_j}}(x^{y^{-1}})=f_1(x^{y^{-1}},y)^{-1}\frac{\partial_{f}}{\partial{x_j}}(x)-f_1(x^{y^{-1}},y)^{-1}f_2(x^{y^{-1}},y)\frac{\partial_{f}}{\partial{x_j}}(y).
\]
for any $x,y\in FQ(S)$.\\
 For a relator $r=(r_1,r_2)$, we define $\frac{\partial_{f}}{\partial{x_j}}(r):=\frac{\partial_{f}}{\partial{x_j}}(r_1)-\frac{\partial_{f}}{\partial{x_j}}(r_2)$.

\begin{prop}
\label{prop:calculation}
Let $Q$ be a quandle with a finite presentation $\langle S\mid R\rangle$, $X$ a quandle, $A$ an abelian group and $\theta:X^2\to A$ a quandle $2$-cocycle. For any element $x^{y_{1}^{\varepsilon_1}\cdots y_{n}^{\varepsilon_n}}\in FQ(S)$ and $z\in S$, we have 
\[
\frac{\partial_{f_{\theta}}}{\partial{z}}(x^{y_{1}^{\varepsilon_1}\cdots y_{n}^{\varepsilon_n}})=1\cdot\left(\sum^{n}_{i=1}\varepsilon_i\theta(x^{y_{1}^{\varepsilon_1}\cdots y_{i-1}^{\varepsilon_{i-1}}y_i^{\frac{\varepsilon_i-1}{2}}},y_i)\right)\frac{\partial_{f_{\theta}}}{\partial{z}}(x).
\]
\end{prop}
\begin{proof}
By the definition of the $f$-derivative, it holds that
\[
\frac{\partial_{f_{\theta}}}{\partial z}(x^y)=f_\theta(x,y)\frac{\partial_{f_{\theta}}}{\partial z}(x)+0(x,y)\frac{\partial_{f_{\theta}}}{\partial z}(y)=1\cdot\theta(x,y)\frac{\partial_{f_{\theta}}}{\partial z}(x)
\]
and
\begin{eqnarray*}
\frac{\partial_{f_{\theta}}}{\partial z}(x^{y^{-1}})&=&f_\theta(x^{y^{-1}},y)^{-1}\frac{\partial_{f_{\theta}}}{\partial z}(x)-f_{\theta}(x^{y^{-1}},y)^{-1}0(x^{y^{-1}},y)\frac{\partial_{f_{\theta}}}{\partial z}(y)\\
&=&1\cdot(-\theta(x^{y^{-1}},y))\frac{\partial_{f_{\theta}}}{\partial z}(x)
\end{eqnarray*}
for any $x,y\in FQ(S)$. Hence, we have
\[
\frac{\partial_{f_{\theta}}}{\partial{z}}(x^{y_{1}^{\varepsilon_1}\cdots y_{n}^{\varepsilon_n}})=1\cdot(\varepsilon_n\theta(x^{y_{1}^{\varepsilon_1}\cdots y_{n-1}^{\varepsilon_{n-1}}y_{n}^{\frac{\varepsilon_n-1}{2}}},y_{n}))\frac{\partial_{f_{\theta}}}{\partial x_j}(x^{y_{1}^{\varepsilon_1}\cdots y_{n-1}^{\varepsilon_{n-1}}}).
\]
Repeating this procedure, we see that
\[
\frac{\partial_{f_{\theta}}}{\partial{z}}(x^{y_{1}^{\varepsilon_1}\cdots y_{n}^{\varepsilon_n}})=1\cdot\left(\sum^{n}_{i=1}\varepsilon_i\theta(x^{y_{1}^{\varepsilon_1}\cdots y_{i-1}^{\varepsilon_{i-1}}y_{i}^{\frac{\varepsilon_i-1}{2}}},y_{i})\right)\frac{\partial_{f_{\theta}}}{\partial z}(x).
\]
\end{proof}

Let $A$ be an $m\times n$ matrix over a commutative ring $R$. The {\it $d$-th elementary ideal} of $A$, which is denoted by $E_d(A)$, is the ideal generated by all $(n-d)$-minors of $A$ if $n-m\leq d<n$, and
\[
E_d(A)=\begin{cases}
0\quad\ \textrm{if }d<n-m,\\
R\quad\textrm{if }n\leq d.
\end{cases}
\] 

 Let $Q$ be a quandle with a finite presentation $\langle x_1,\ldots,x_n\mid r_1,\ldots,r_m\rangle$, $X$ a quandle and $\rho:Q\to X$ a quandle homomorphism. Let $R$ be a ring and $f=(f_1,f_2)$ an Alexander pair of maps $f_1,f_2:X\times X\to R$. We remark that the pair $f\circ\rho^2=(f_1\circ(\rho\times\rho),f_2\circ(\rho\times\rho))$ is also an Alexander pair (see \cite{Ishii2022twisted}). Let us define the $m\times n$ matrix $A(Q,\rho;f_1,f_2)$ by
\[
A(Q,\rho;f_1,f_2)=\left(
\begin{array}{ccc}
\frac{\partial_{f\circ\rho^2}}{\partial{x_1}}(r_1) & \cdots & \frac{\partial_{f\circ\rho^2}}{\partial{x_n}}(r_1) \\
\vdots & \ddots & \vdots \\
\frac{\partial_{f\circ\rho^2}}{\partial{x_1}}(r_m) & \cdots & \frac{\partial_{f\circ\rho^2}}{\partial{x_n}}(r_m)
\end{array}
\right).
\]
We call this matrix $A(Q,\rho;f_1,f_2)$ {\it Alexander matrix of the finite presentation $\langle x_1,\ldots,x_n\mid r_1,\ldots,r_m\rangle$ associated to the quandle homomorphism $\rho$} or the {\it $f$-twisted Alexander matrix} of $(Q,\rho)$ with respect to the quandle presentation $\langle x_1,\ldots,x_n\mid r_1,\ldots,r_m\rangle$ \cite{Ishii2022twisted}. 

%Let $\psi:R^m\to R^n$ be an $R$-module homomorphism defined by $\psi({\bm u})={\bm u}A(Q,\rho;f_1,f_2)$. We denote the Cokernel of $\psi$ by ${\rm Coker}(A(Q,\rho;f_1,f_2))$.

Let $Q^{\prime}$  be a quandle with a finite presentation and $\rho^{\prime}:Q^{\prime}\to X$ be a quandle homomorphism. In \cite{Ishii2022twisted}, A. Ishii and K. Oshiro showed that if there exists a quandle isomorphism $\varphi:Q\to Q^{\prime}$ such that $\rho=\rho^{\prime}\circ\varphi$, then $A(Q,\rho;f_1,f_2)$ and $A(Q^{\prime},\rho^{\prime};f_1,f_2)$ are related by a finite sequence of the following transformations $({\rm M1})\sim({\rm M8})$:
\begin{align*}
&({\rm M1})\ ({\bm a_1},\ldots,{\bm a_i},\ldots,{\bm a_j},\ldots,{\bm a_n})\leftrightarrow({\bm a_1},\ldots,{\bm a_i}+{\bm a_j}r,\ldots,{\bm a_j},\ldots,{\bm a_n})\ (r\in R),\\
&({\rm M2})\ \left(
\begin{array}{c}
{\bm a_1} \\
\vdots \\
{\bm a_i}\\
\vdots \\
{\bm a_j} \\
\vdots \\
{\bm a_n}
\end{array}
\right)\leftrightarrow\left(
\begin{array}{c}
{\bm a_1} \\
\vdots \\
{\bm a_i}+r{\bm a_j} \\
\vdots \\
{\bm a_j} \\
\vdots \\
{\bm a_n}
\end{array}
\right)\ (r\in R),\\ 
&({\rm M3})\ A\leftrightarrow\left(
\begin{array}{c}
A\\
{\bm 0}
\end{array}
\right),\ \ 
({\rm M4})\ A\leftrightarrow\left(
\begin{array}{cc}
A & {\bm 0}\\
{\bm 0} & 1
\end{array}
\right).\\
& ({\rm M5})  ({\bm a_1},\ldots,{\bm a_i},\ldots,{\bm a_j},\ldots,{\bm a_n})\leftrightarrow({\bm a_1},\ldots,{\bm a_j},\ldots,{\bm a_i},\ldots,{\bm a_n}),\\
& ({\rm M6}) ({\bm a_1},\ldots,{\bm a_i},\ldots,{\bm a_n})\leftrightarrow({\bm a_1},\ldots,{\bm a_i}u,\ldots,{\bm a_n})\ (u:\textrm{a unit of }R),\\
&({\rm M7})\ \left(
\begin{array}{c}
{\bm a_1} \\
\vdots \\
{\bm a_i}\\
\vdots \\
{\bm a_j} \\
\vdots \\
{\bm a_n}
\end{array}
\right)\leftrightarrow\left(
\begin{array}{c}
{\bm a_1} \\
\vdots \\
{\bm a_j} \\
\vdots \\
{\bm a_i} \\
\vdots \\
{\bm a_n}
\end{array}
\right),\ \ ({\rm M8})\ \left(
\begin{array}{c}
{\bm a_1} \\
\vdots \\
{\bm a_i}\\
\vdots \\
{\bm a_n}
\end{array}
\right)\leftrightarrow\left(
\begin{array}{c}
{\bm a_1} \\
\vdots \\
u{\bm a_i}\\
\vdots \\
{\bm a_n}
\end{array}
\right)\ (u:\textrm{a unit of }R).
\end{align*} 

Furthermore, if $R$ is a commutative ring, we have 
\[
E_d(A(Q,\rho;f_1,f_2))=E_d(A(Q^{\prime},\rho^{\prime};f_1,f_2)).
\]
\begin{comment}
\begin{exam}[\cite{Ishii2022twisted}]
Let $K$ be an oriented classical knot, $X$ be a quandle and $(f_1,f_2)$ be the Alexander pair of maps $f_1,f_2:X^2\to\Z[t^{\pm 1}]$ in Example \ref{exam:Alexander_pair_1}. 

In this setting, for any quandle homomorphism $\rho:Q(K)\to X$, the $f$-twisted Alexander matrix $A(Q(K),\rho;f_1,f_2)$ is an {\it Alexander matrix} (cf. \cite{Burde2014knots}). Thus, we see that ${\rm Coker}(A(Q(K),\rho;f_1,f_2))$ is isomorphic to the direct sum of the {\it Alexander module} of $K$ and $\Z$ and it holds that $E_d(A(Q(K),\rho;f_1,f_2))$ is equal to the {\it $d-1$-th elementary ideal} of $K$ (cf. \cite{Burde2014knots}).
\end{exam}
\end{comment}
 
\section{Surface knot invariants obtained from a quandle $2$-cocycle}
\label{sect:Alexander_matrix_surface_knot}
%In \cite{Taniguchitwisted}, the author showed that the $0$-th elementary ideal of an $f$-twisted Alexander matrix of knot quandles  of a classical knot using an Alexander pair in Example \ref{exam:Alexander_pair_2} is determined by the {\it quandle cocycle invariants} \cite{Carter2003quandle}. 
In this section, we will discuss the relationship between $f$-twisted Alexander matrices of the knot quandles of surface knots using an Alexander pair in Example \ref{exam:Alexander_pair_2} and the Carter-Saito-Satoh's invariant \cite{Carter2006ribbon}. At first, we will review definitions of surface knots and its diagrams. Refer to \cite{Carter2004surfaces,Kamada2017surface} for details

A {\it surface knot} is a connected oriented closed surface smoothly embedded in 4-space $\R^4$ up to ambient isotopies. When the surface is homeomorphic to a 2-sphere, it is called a {\it 2-knot}. We fix a projection ${\rm pr}:\R^4\to \R^3$. Every surface knot can be perturbed slightly in $\R^4$ so that the singularity set of ${\rm pr}(F)$ consists of double points, isolated triple points, and isolated branched points as illustrated in Figure \ref{singular_set}. Each connected component of the set of all double points of ${\rm pr}(F)$ is a curve in $\R^3$, which is called a {\it double point curve}. For each double point curve $\delta$, the preimage ${\rm pr}^{-1}(\delta)$ consists a union of two curves in $F$. Then, we call the curves which is under the other the {\it lower decker curve}. 

\begin{figure}[h]
	 \centering{\includegraphics[height=2.5cm]{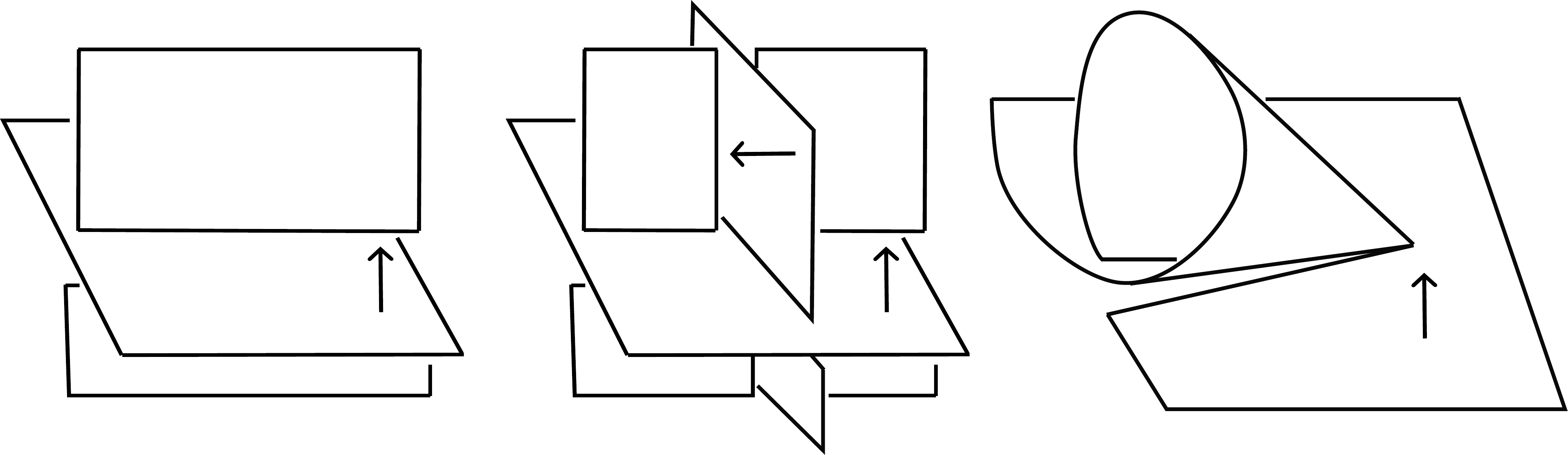}}
	\caption{The singular points}
	\label{singular_set}
 \end{figure}

The crossing information is indicated in ${\rm pr}(F)$ as follows: At each double point curve, two sheets intersects locally, one of which is under the other relative to the projection direction of $p$. Then, the lower sheet is broken by the upper sheet. A {\it diagram} of $F$ is the image ${\rm pr}(F)$ with such crossing information. We can regard a diagram as a union of disjoint compact, connected, surfaces. For a diagram $D$, we denote by $S(D)$ the set of such connected surfaces of $D$. We call an element of $S(D)$ a {\it sheet}. Since the surface is oriented, we take normal vectors $\vec{n}$ to ${\rm pr}(F)$ such that the triple $(\vec{v_1},\vec{v_2},\vec{n})$ represents the orientation of $\R^3$, where $(\vec{v_1},\vec{v_2})$ defines the orientation of ${\rm pr}(F)$. Such normal vectors are defined on the ${\rm pr}(F)$ at all points othar than isolated branched points. We call the normal orientation represented by $\vec{n}$ is called the {\it normal orientation} determined from the orientation of $F$. In this paper, we indicate the orientations of sheets by the normal orientation.

Let $D$ be a diagram of a surface knot $F$. For each double point curve $\delta$, let $x_i,x_j,x_k$ be sheets around the double point curve $\delta$ such that the normal orientation of the upper sheet $x_j$ points from $x_i$ to $x_k$ as shown in Figure \ref{coloring_rule}. 
\begin{figure}[h]
	 \centering{\includegraphics[height=3cm]{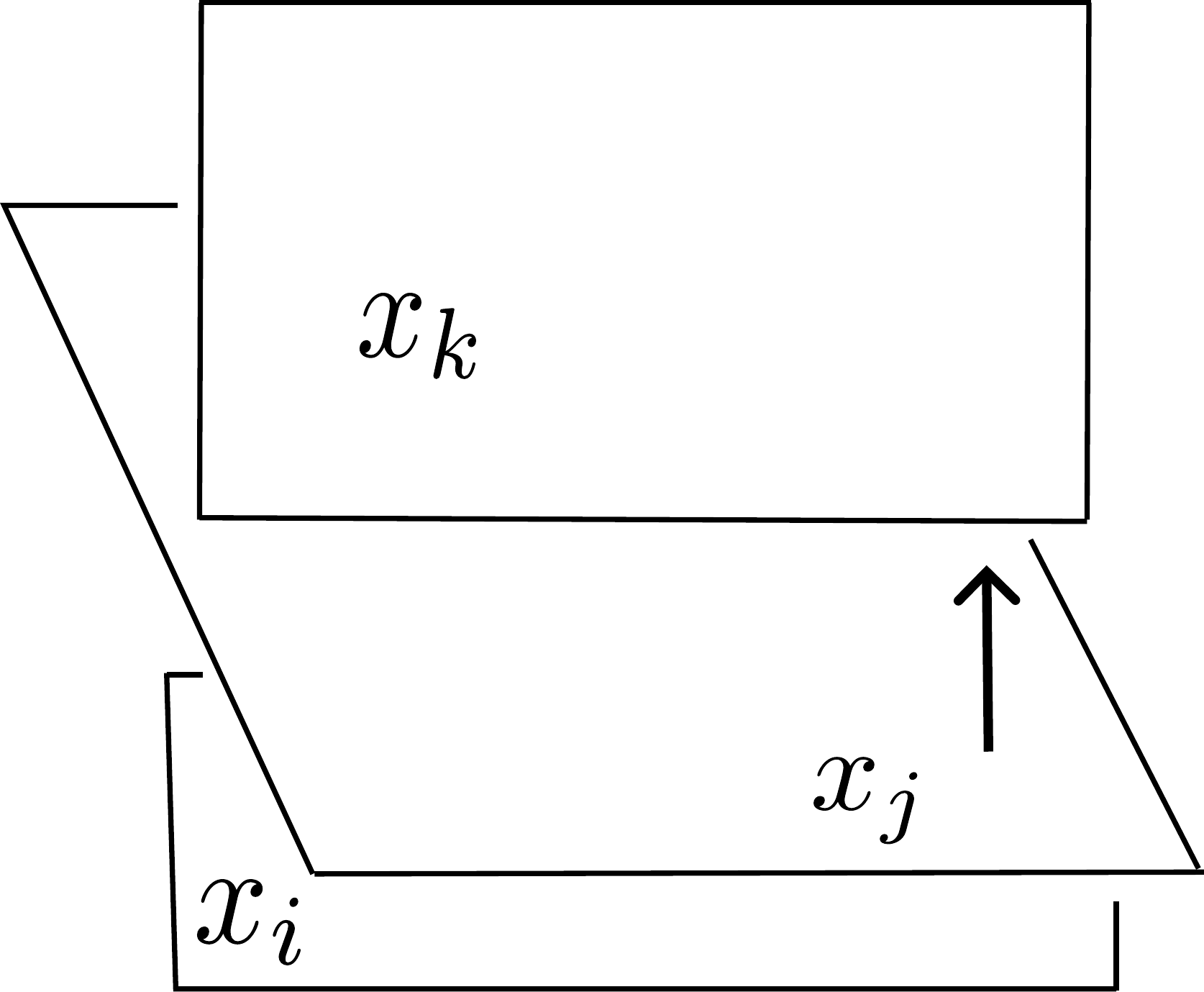}}
	\caption{The sheets around double point curve $\delta$}
	\label{coloring_rule}
 \end{figure}
The sheet $x_i$ is called the {\it source sheet} of $\delta$. Then, we define the relator $r_\delta$ by $(x_i\ast x_j,x_k)$. The knot quandle of a surface knot $F$ has the following presentation:
\[
\langle S(D)\mid \{ r_{\delta}\mid \delta:\textrm{ a double point curve of }D\}\rangle.
\]
This presentaion is called the {\it Wirtinger presentation} of $Q(F)$ with respect to the diagram $D$.

Next, we will introduce the Carter-Saito-Satoh's invariant \cite{Carter2006ribbon}. Let $X$ be a quandle. A map $c:S(D)\to X$ is an {\it $X$-coloring of $D$} if it satisfies the following condition near each double point curve: if $x_i,x_j,x_k$ are the sheets as illustrated in Figure \ref{coloring_rule}, then it holds that $c(x_i)\ast c(x_j)=c(x_k)$. The element $c(x)$ assigned to the sheet $x$ is called the {\it color} of $x$.  We denote the set of all $X$-colorings of $D$ by ${\rm Col}_X(D)$. 

Let $\rho:Q(F)\to X$ be a quandle homomorphism. Using the Wirtinger presentation of $Q(F)$ with respect to the diagram $D$, we can regard each sheet as an element of $Q(F)$. Hence, let us define the map $c_{\rho}:S(D)\to X$ by $c_{\rho}(x)=\rho(x)$ for any $x\in S(D)$. Then, the map $c_{\rho}$ is an $X$-coloring of $D$. In this paper, we call $c_{\rho}$ is the {\it $X$-coloring of $D$ corresponding to $\rho$}. It is known that the map ${\rm Hom}(Q(F),X)\to{\rm Col}_X(D);\rho\mapsto c_\rho$ is bijective.

Let $c:S(D)\to X$ be an $X$-coloring of $D$ and $\theta:X^2\to A$ be a quandle $2$-cocycle. We consider an oriented immersed circle $L$ on $F$. We denote ${\rm pr}(L)$ by $L^D$. We assume that $L^D$ intersects the double point curves transversely, and misses triple points and branched points. Let $d_1,\ldots,d_n$ be points on $F$ at which $L$ intersects the lower decker curves. In this paper, we denote ${\rm pr}(x)$ by $x^D$ for $x\in F$. For each $d_l$, we give the sign $\varepsilon(d_{l})\in\{\pm 1\}$ such that $\varepsilon(d_{l})=+1$ if and only if the orientation of $L^D$ agrees with the normal orientation of the upper sheet around $d_{l}^D$. Then, let us define an element $W_{\theta}(d_l,c)$ as follows: Let $x_i,x_j,x_k$ be the sheets around $d_l^D$ such that $x_j$ is the upper sheet and $x_i$ is the source sheet as shown in Figure \ref{weight_surface_knot}. We set $W_{\theta}(d_l,c):=\varepsilon(d_{l})\theta(c(x_i),c(x_j))$. Moreover, we put $W_{\theta}(L,c):=\sum^n_{l=1}W_{\theta}(d_l,c)$. 
\begin{figure}[h]
	 \centering{\includegraphics[height=4cm]{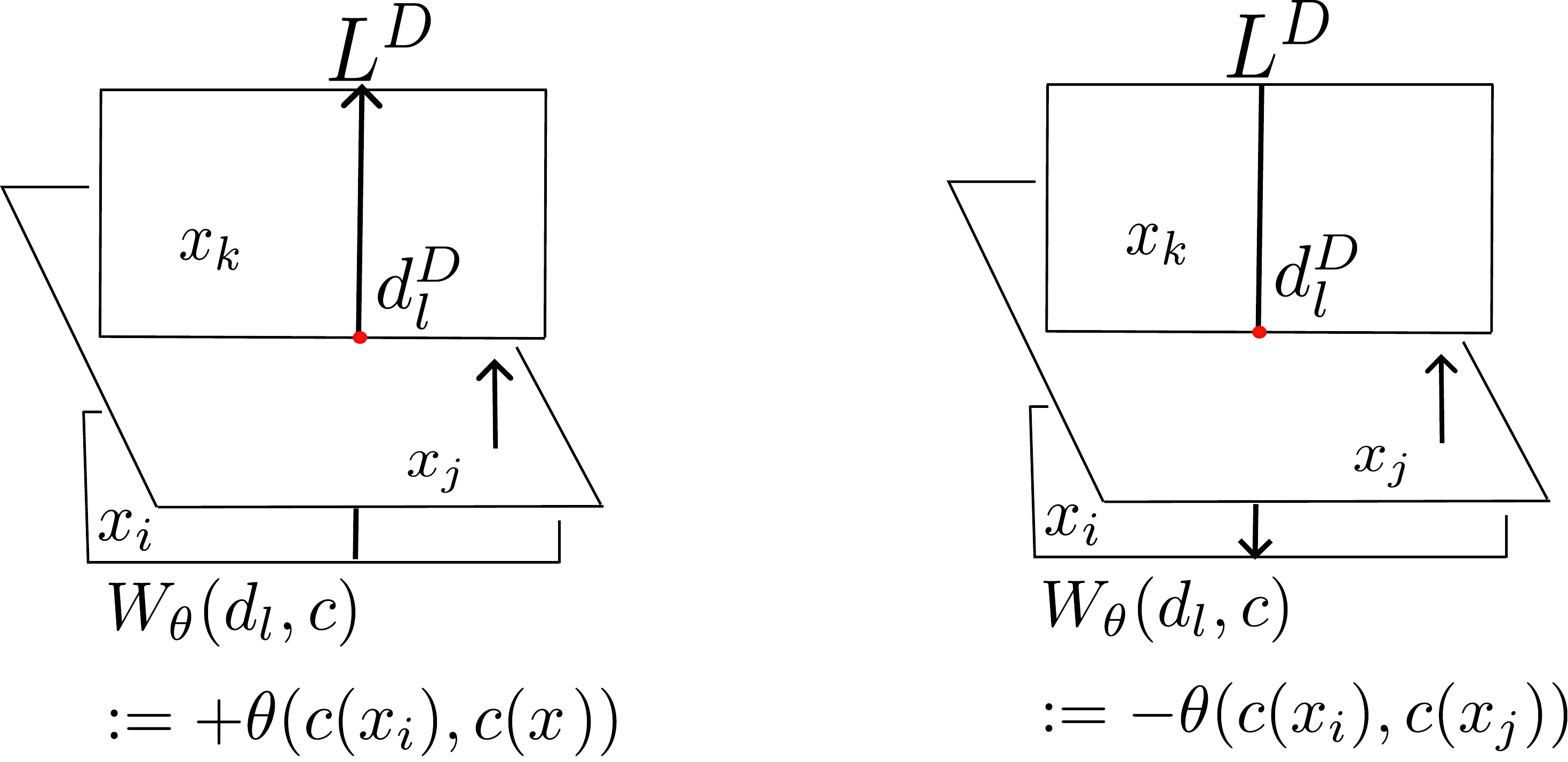}}
	\caption{The weight at $d_l$}
	\label{weight_surface_knot}
 \end{figure}
\begin{lemm}[\cite{Carter2006ribbon}]
If $L$ and $L^{\prime}$ are homologous, then we have $W_{\theta}(L,c)=W_{\theta}(L^{\prime},c)$.
\end{lemm}
Hence, for each $\lambda\in H_1(F)$, we can define $W_{\theta}(\lambda,c):=W_{\theta}(L_\lambda,c)$, where $L_\lambda$ is a representative curve of $\lambda$. Then, we denote a multi-set  $\{W_{\theta}(\lambda,c)\mid c\in{\rm Col}_{X}(D)\}$ by $\Omega_{\theta}(\lambda)$. Furthermore, we define a family of multi-sets of $A$ by $\Omega_{\theta}(F)=\{\Omega_{\theta}(\lambda)\mid\lambda\in H_1(F)\}$.
\begin{prop}[\cite{Carter2006ribbon}]
The family $\Omega_{\theta}(F)$ does not depend on the choice of a diagram $D$ of $F$.
\end{prop}
Thus, $\Omega_{\theta}(F)$ is an invariant of surface knots. The aim of this section is to show the following Theorem.
\begin{theo}
\label{theo:main_1}
Let $D$ be a diagram of surface knot $F$, $X$ a quandle, $A$ an abelian group and $\theta:X^2\to A$ a quandle 2-cocycle. Then, for any quandle homomorphism $\rho:Q(F)\to X$, it holds that
\[
(\{1\cdot W_{\theta}(\lambda,c_{\rho})-1\cdot 0_A\mid \lambda\in H_1(F;\Z)\})=E_0(A(Q(F),\rho;f_{\theta},0)),
\]
where $c_{\rho}$ is the $X$-coloring of $D$ corresponding to $\rho$.
\end{theo}
\begin{remark}
\label{remark:ideal}
{\rm
Let $\{ b_1,\ldots,b_{2g}\}$ be a basis of $H_1(F)$. Since $W_{\theta}(\lambda_1+\lambda_2,c)$ is equal to $W_{\theta}(\lambda_1,c)+W_{\theta}(\lambda_2,c)$ for any $\lambda_1,\lambda_2\in H_1(F)$, we have
\begin{eqnarray*}
1\cdot W_{\theta}(\lambda_1+\lambda_2,c)-1\cdot 0_A&=&1\cdot( W_{\theta}(\lambda_1,c)+W_{\theta}(\lambda_2,c))-1\cdot 0_A\\
&=&1\cdot W_{\theta}(\lambda_1,c)(1\cdot W_{\theta}(\lambda_2,c)-1\cdot 0_A)\\
&&+1\cdot W_{\theta}(\lambda_1,c)-1\cdot 0_A.
\end{eqnarray*}
 Hence, the ideal generated by $\{1\cdot W_{\theta}(\lambda,c)-1\cdot 0_A\mid \lambda\in H_1(F;\Z)\}$ coincides with the ideal $(1\cdot W_{\theta}(b_1,c)-1\cdot 0_A,\ldots,1\cdot W_{\theta}(b_{2g},c)-1\cdot 0_A)$. %This implies that the ideal $(\{1\cdot W_{\theta}(\lambda,c)-1\cdot 0_A\mid \lambda\in H_1(F;\Z)\})$ is completely determined by $W_{\theta}(b_1,c),\ldots,W_{\theta}(b_{2g},x)$
}
\end{remark}
\begin{proof}[Proof of Theorem \ref{theo:main_1}]
Let $D$ be a diagram of a surface knot $F$, $S(D)=\{ x_1,\ldots,x_n\}$ the set of all sheets of $D$, ${\rm Ldeck(D)}$ the set of all lower decker curves and $q$ a point in $x_n$. For each $i\in\{1,\ldots,n-1\}$, we take a path $\gamma_i$ in $F$ starting from ${\rm pr}^{-1}(q)$ and terminating at a point of ${\rm pr}^{-1}(x_i)$ as shown in Figure \ref{path_gamma}. We assume that ${\rm pr}(\gamma_i)$ intersects the double point curve transversely, and misses triple points and branched points. We denote the path $[0,1]\to F;x\to q$ by $\gamma_n$. For each $i\in\{1,\ldots,n\}$, let $m_i$ be the cardinality of the set $\gamma_i\cap{\rm Ldeck}(D)$ and $d_{i1},\ldots,d_{im_i}$ the points at which $\gamma_i$ intersects ${\rm Ldeck}(D)$ in turn with the orientation. We denote the upper-sheet at $d_{ij}^D$ by $u_{ij}$. Here, $u_{ij}$ is an element of $S(D)$. Then, let us define the relator $r_i$ by $(x_n^{u_{i1}^{\varepsilon(d_{i1})}\cdots u_{im_i}^{\varepsilon(d_{im_i})}},x_i)$. 
\begin{figure}[h]
	 \centering{\includegraphics[height=3cm]{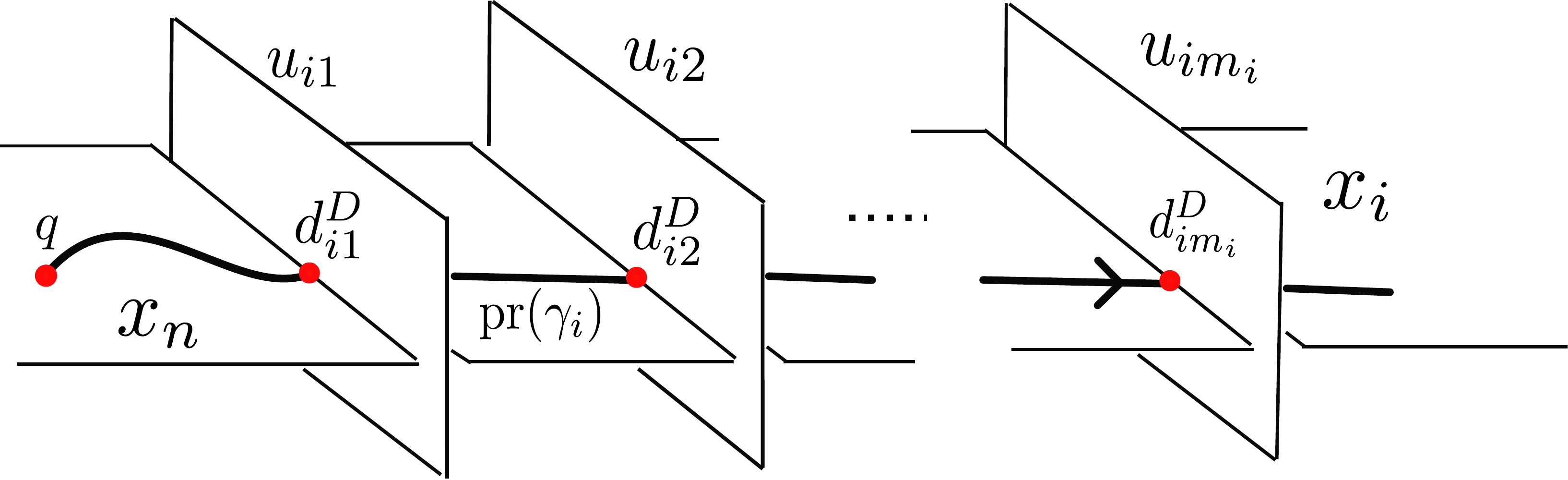}}
	\caption{The path $\gamma_i$}
	\label{path_gamma}
 \end{figure}

 %For each double point curve $\delta$, let $x_i,x_j$ be the lower sheets around $\delta$. We denote the upper sheet separating $x_i$ and $x_j$ by $u_\delta$. Replacing index if necessary, we may assume that the orientation of $u_\delta$ is directed from $x_i$ to $x_j$. Let us define the relator $s_\delta$ by $(x_n^{u_{i1}^{\varepsilon(d_{i1})}\cdots u_{im_i}^{\varepsilon(d_{im_i})}u_\delta u_{jm_j}^{-\varepsilon(d_{jm_j})}\cdots u_{j1}^{-\varepsilon(d_{j1})}},x_n)$. Using Tietze's moves, we see that 

We denote the upper sheet around a double point curve $\delta$ by $u_\delta$. For each double point curve $\delta$, let $s(\delta)$ be an element of $\{1,\ldots,n\}$ satisfying that $x_{s(\delta)}$ is the source sheet of $\delta$, and $t(\delta)$ an element of $\{1,\ldots,n\}$ satisfying that the normal orientation of the upper-sheet $u_\delta$ points to $x_{t(\delta)}$. Then, we define the relator $r_\delta$ by 
\[
r_\delta=\left( x_n^{u_{s(\delta)1}^{\varepsilon\left(d_{s(\delta)1}\right)}\cdots u_{s(\delta)m_{s(\delta)}}^{\varepsilon\left(d_{s(\delta)m_{s(\delta)}}\right)}u_\delta u_{t(\delta)m_{t(\delta)}}^{-\varepsilon\left(d_{t(\delta)m_{t(\delta)}}\right)}\cdots u_{t(\delta)1}^{-\varepsilon\left(d_{t(\delta)1}\right)}},x_n\right).
\]
 We see that 
$\left\langle x_1,\ldots x_n~
\begin{array}{|c}
r_1,\ldots,r_{n-1}\\
r_\delta \ (\delta:\textrm{a double point curve})
\end{array}
\right\rangle
$ and the Wirtinger presentation of $D$ are related by a finite sequence of the Tietze's moves. Thus, $\left\langle x_1,\ldots x_n~
\begin{array}{|c}
r_1,\ldots,r_{n-1}\\
r_\delta \ (\delta:\textrm{a double point curve})
\end{array}
\right\rangle
$ is also a presentation of the knot quandle $Q(F)$. We fix this presentation. Let $\rho:Q(F)\to X$ be a quandle homomorphism and $c_{\rho}:S(D)\to X$ the $X$-coloring of $D$ corresponding to $\rho$. Then, we consider the $f$-twisted Alexander matrix of $(Q(F),\rho)$ with resepect to the fixed presentation. 
%Since the proof of Theorem \ref{theo:main_1}, $A(Q(F),\rho;f_{\theta},0)$ is equivalent to the $|\A|\times 1$ matrix $\left(\frac{\partial_{f\circ\rho}}{\partial{x_n}}(s_\al)\right)_{\al\in \A}$. 

At first, let us consider $\frac{\partial_{f_{\theta}\circ\rho^2}}{\partial{x_j}}(r_i)$. By Proposition \ref{prop:calculation}, for any $i\in\{1,\ldots,n-1\}$ and $j\in\{1,\ldots,n\}$, there is an element $a_i\in A$ such that 
\[
\frac{\partial_{f_{\theta}\circ\rho^2}}{\partial{x_j}}\left(x_n^{u_{i1}^{\varepsilon(d_{i1})}\cdots u_{im_i}^{\varepsilon(d_{im_i})}}\right)=1\cdot a_i\frac{\partial_{f_{\theta}\circ\rho^2}}{\partial{x_j}}(x_n).
\] 
Hence, we have the following equalities:
\begin{eqnarray*}
\frac{\partial_{f_{\theta}\circ\rho^2}}{\partial{x_j}}(r_i)&=&\frac{\partial_{f_{\theta}\circ\rho^2}}{\partial{x_j}}\left(x_n^{u_{i1}^{\varepsilon(d_{i1})}\cdots u_{im_i}^{\varepsilon(d_{im_i})}}\right)-\frac{\partial_{f_{\theta}\circ\rho^2}}{\partial{x_j}}(x_i)\\
&=&1\cdot a_i\frac{\partial_{f_{\theta}\circ\rho^2}}{\partial{x_j}}(x_n)-\frac{\partial_{f_{\theta}\circ\rho^2}}{\partial{x_j}}(x_i)\\
&=&\begin{cases}
1\cdot a_i\hspace{5.5mm} (j=n)\\
-1\cdot 0_A\ (j=i)\\
0\hspace{11.5mm} (j\neq i,n).
\end{cases}
\end{eqnarray*}

Next, we will discuss $\frac{\partial_{f_{\theta}\circ\rho^2}}{\partial{x_j}}(r_\delta)$ for any double point curve $\delta$ and $j\in\{1,\ldots,n\}$. For each double point curve $\delta$, we define the immersed curve $L_\delta$ on $F$ by $L_\delta=\gamma_{s(\delta)}\cdot\gamma_\delta\cdot\gamma_{t(\delta)}^{-1}$, where $\gamma_\delta$ is a path from the terminal point of $\gamma_{s(\delta)}$ to the terminal point of $\gamma_{t(\delta)}$ which intersects lower decker curves only at point $d_\delta$ as illustrated in Figure \ref{path_alpha}.
\begin{figure}[h]
	 \centering{\includegraphics[height=4.5cm]{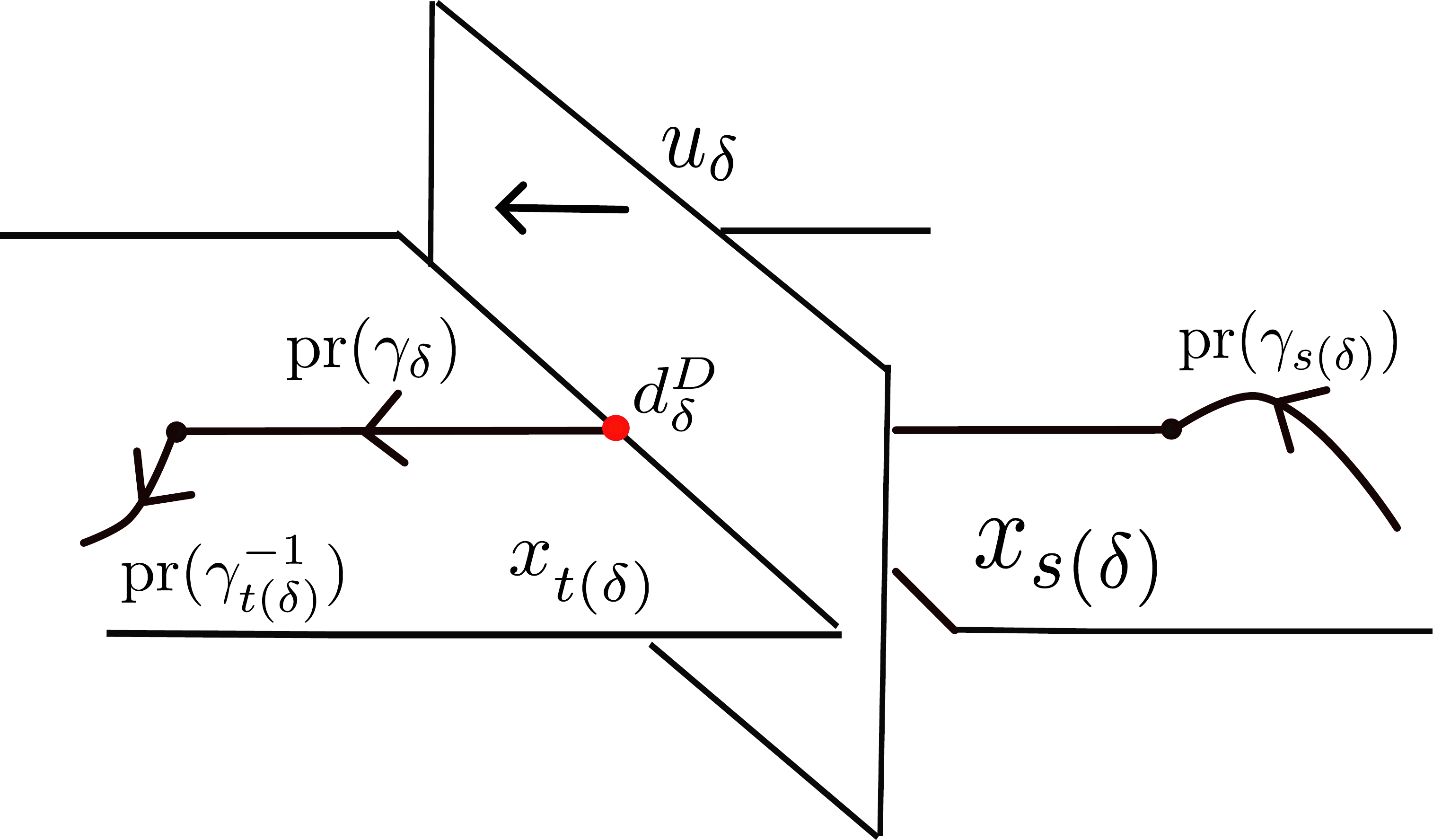}}
	\caption{The path $\gamma_\delta$}
	\label{path_alpha}
 \end{figure}

 We remark that $L_\delta\cap {\rm Ldeck}(D)=\{d_{s({\delta})1},\ldots,d_{s({\delta})m_{s(\delta)}},d_\delta,d_{t(\delta)1},\ldots,d_{t(\delta)m_{t(\delta)}}\}$. If $\varepsilon(d_{t(\delta)1})=+1$, it holds that
\begin{eqnarray*}
&&\frac{\partial_{f_{\theta}\circ\rho^2}}{\partial{x_j}}\left( x_n^{u_{s(\delta)1}^{\varepsilon\left(d_{s(\delta)1}\right)}\cdots u_{s(\delta)m_{s(\delta)}}^{\varepsilon\left(d_{s(\delta)m_{s(\delta)}}\right)}u_\delta u_{t(\delta)m_{t(\delta)}}^{-\varepsilon\left(d_{t(\delta)m_{t(\delta)}}\right)}\cdots u_{t(\delta)1}^{-\varepsilon\left(d_{t(\delta)1}\right)}}\right)\\
&&\frac{\partial_{f_{\theta}\circ\rho^2}}{\partial{x_j}}\left( x_n^{u_{s(\delta)1}^{\varepsilon\left(d_{s(\delta)1}\right)}\cdots u_{s(\delta)m_{s(\delta)}}^{\varepsilon\left(d_{s(\delta)m_{s(\delta)}}\right)}u_\delta u_{t(\delta)m_{t(\delta)}}^{-\varepsilon\left(d_{t(\delta)m_{t(\delta)}}\right)}\cdots u_{t(\delta)1}^{-1}}\right)\\
&=&1\cdot\left(-\theta\left(\rho\left(x_n^{u_{s(\delta)1}^{\varepsilon\left(d_{s(\delta)1}\right)}\cdots u_{s(\delta)m_{s(\delta)}}^{\varepsilon\left(d_{s(\delta)m_{s(\delta)}}\right)}u_\delta u_{t(\delta)m_{t(\delta)}}^{-\varepsilon\left(d_{t(\delta)m_{t(\delta)}}\right)}\cdots u_{t(\delta)1}^{-1}}\right),\rho(u_{t(\delta)1})\right)\right)\\
&&\times\frac{\partial_{f_{\theta}\circ\rho^2}}{\partial{x_j}}\left( x_n^{u_{s(\delta)1}^{\varepsilon\left(d_{s(\delta)1}\right)}\cdots u_{s(\delta)m_{s(\delta)}}^{\varepsilon\left(d_{s(\delta)m_{s(\delta)}}\right)}u_\delta u_{t(\delta)m_{t(\delta)}}^{-\varepsilon\left(d_{t(\delta)m_{t(\delta)}}\right)}\cdots u_{t(\delta)2}^{-\varepsilon\left(d_{t(\delta)2}\right)}}\right).
\end{eqnarray*}
By the definition of $c_{\rho}$, we remark that the following conditions are satisfied:
\begin{itemize}
\item The element $\rho(u_{t(\delta)1})$ equals the color of the upper sheet around $d_{t(\delta)1}^D$ by $c_\rho$
\item The element $\rho\left(x_n^{u_{s(\delta)1}^{\varepsilon\left(d_{s(\delta)1}\right)}\cdots u_{s(\delta)m_{s(\delta)}}^{\varepsilon\left(d_{s(\delta)m_{s(\delta)}}\right)}u_\delta u_{t(\delta)m_{t(\delta)}}^{-\varepsilon\left(d_{t(\delta)m_{t(\delta)}}\right)}\cdots u_{t(\delta)1}^{-1}}\right)$ equals the color of the source sheet around $d_{t(\delta)1}^D$ by $c_{\rho}$. 
\end{itemize}
Thus, it holds that 
\[
-\theta\left(\rho\left(x_n^{u_{s(\delta)1}^{\varepsilon\left(d_{s(\delta)1}\right)}\cdots u_{s(\delta)m_{s(\delta)}}^{\varepsilon\left(d_{s(\delta)m_{s(\delta)}}\right)}u_\delta u_{t(\delta)m_{t(\delta)}}^{-\varepsilon\left(d_{t(\delta)m_{t(\delta)}}\right)}\cdots u_{t(\delta)1}^{-1}}\right),\rho(u_{t(\delta)1})\right)=W_{\theta}(d_{t(\delta)1},c_{\rho}).\]
 Hence, we have the following equality:
\begin{eqnarray*}
&&\frac{\partial_{f_{\theta}\circ\rho^2}}{\partial{x_j}}\left( x_n^{u_{s(\delta)1}^{\varepsilon\left(d_{s(\delta)1}\right)}\cdots u_{s(\delta)m_{s(\delta)}}^{\varepsilon\left(d_{s(\delta)m_{s(\delta)}}\right)}u_\delta u_{t(\delta)m_{t(\delta)}}^{-\varepsilon\left(d_{t(\delta)m_{t(\delta)}}\right)}\cdots u_{t(\delta)1}^{-\varepsilon\left(d_{t(\delta)1}\right)}}\right)\\
&=&1\cdot W_{\theta}(d_{t(\delta)1},c_{\rho})\frac{\partial_{f_{\theta}\circ\rho^2}}{\partial{x_j}}\left( x_n^{u_{s(\delta)1}^{\varepsilon\left(d_{s(\delta)1}\right)}\cdots u_{s(\delta)m_{s(\delta)}}^{\varepsilon\left(d_{s(\delta)m_{s(\delta)}}\right)}u_\delta u_{t(\delta)m_{t(\delta)}}^{-\varepsilon\left(d_{t(\delta)m_{t(\delta)}}\right)}\cdots u_{t(\delta)2}^{-\varepsilon\left(d_{t(\delta)2}\right)}}\right).
\end{eqnarray*}

If $\varepsilon(d_{t(\delta)1})=-1$, we see that
\begin{eqnarray*}
&&\frac{\partial_{f_{\theta}\circ\rho^2}}{\partial{x_j}}\left( x_n^{u_{s(\delta)1}^{\varepsilon\left(d_{s(\delta)1}\right)}\cdots u_{s(\delta)m_{s(\delta)}}^{\varepsilon\left(d_{s(\delta)m_{s(\delta)}}\right)}u_\delta u_{t(\delta)m_{t(\delta)}}^{-\varepsilon\left(d_{t(\delta)m_{t(\delta)}}\right)}\cdots u_{t(\delta)1}^{-\varepsilon\left(d_{t(\delta)1}\right)}}\right)\\
&&\frac{\partial_{f_{\theta}\circ\rho^2}}{\partial{x_j}}\left( x_n^{u_{s(\delta)1}^{\varepsilon\left(d_{s(\delta)1}\right)}\cdots u_{s(\delta)m_{s(\delta)}}^{\varepsilon\left(d_{s(\delta)m_{s(\delta)}}\right)}u_\delta u_{t(\delta)m_{t(\delta)}}^{-\varepsilon\left(d_{t(\delta)m_{t(\delta)}}\right)}\cdots u_{t(\delta)1}}\right)\\
&=&1\cdot\left(\theta\left(\rho\left(x_n^{u_{s(\delta)1}^{\varepsilon\left(d_{s(\delta)1}\right)}\cdots u_{s(\delta)m_{s(\delta)}}^{\varepsilon\left(d_{s(\delta)m_{s(\delta)}}\right)}u_\delta u_{t(\delta)m_{t(\delta)}}^{-\varepsilon\left(d_{t(\delta)m_{t(\delta)}}\right)}\cdots u_{t(\delta)2}^{-\varepsilon\left(d_{t(\delta)2}\right)}}\right),\rho(u_{t(\delta)1})\right)\right)\\
&&\times\frac{\partial_{f_{\theta}\circ\rho^2}}{\partial{x_j}}\left( x_n^{u_{s(\delta)1}^{\varepsilon\left(d_{s(\delta)1}\right)}\cdots u_{s(\delta)m_{s(\delta)}}^{\varepsilon\left(d_{s(\delta)m_{s(\delta)}}\right)}u_\delta u_{t(\delta)m_{t(\delta)}}^{-\varepsilon\left(d_{t(\delta)m_{t(\delta)}}\right)}\cdots u_{t(\delta)2}^{-\varepsilon\left(d_{t(\delta)2}\right)}}\right).
\end{eqnarray*}
As in the case of $\varepsilon(d_{t(\delta)1})=1$, it holds that
\begin{eqnarray*}
&&\frac{\partial_{f_{\theta}\circ\rho^2}}{\partial{x_j}}\left( x_n^{u_{s(\delta)1}^{\varepsilon\left(d_{s(\delta)1}\right)}\cdots u_{s(\delta)m_{s(\delta)}}^{\varepsilon\left(d_{s(\delta)m_{s(\delta)}}\right)}u_\delta u_{t(\delta)m_{t(\delta)}}^{-\varepsilon\left(d_{t(\delta)m_{t(\delta)}}\right)}\cdots u_{t(\delta)1}^{-\varepsilon\left(d_{t(\delta)1}\right)}}\right)\\
&=&1\cdot W_{\theta}(d_{t(\delta)1},c_{\rho})\frac{\partial_{f_{\theta}\circ\rho^2}}{\partial{x_j}}\left( x_n^{u_{s(\delta)1}^{\varepsilon\left(d_{s(\delta)1}\right)}\cdots u_{s(\delta)m_{s(\delta)}}^{\varepsilon\left(d_{s(\delta)m_{s(\delta)}}\right)}u_\delta u_{t(\delta)m_{t(\delta)}}^{-\varepsilon\left(d_{t(\delta)m_{t(\delta)}}\right)}\cdots u_{t(\delta)2}^{-\varepsilon\left(d_{t(\delta)2}\right)}}\right).
\end{eqnarray*}
Repeating this procedure, we have
\begin{eqnarray*}
&&\frac{\partial_{f_{\theta}\circ\rho^2}}{\partial{x_j}}(x_n^{u_{i1}^{\varepsilon(d_{i1})}\cdots u_{im_i}^{\varepsilon(d_{im_i})}u_\delta u_{jm_j}^{-\varepsilon(d_{jm_j})}\cdots u_{j1}^{-\varepsilon(d_{j1})}})\\
&=&1\cdot\left(\sum_{d\in L_\delta\cap {\rm Ldeck}(D)}W_{\theta}(d,c_{\rho})\right)\frac{\partial_{f_{\theta}\circ\rho^2}}{\partial{x_j}}(x_n)\\
&=&\begin{cases}
1\cdot W_{\theta}(L_\delta,c_{\rho})\ \ (j=n)\\
0\quad\quad\quad\quad\quad\quad(j\neq n)
\end{cases}.
\end{eqnarray*}
Hence, it holds that 
\[
\frac{\partial_{f_{\theta}\circ\rho}}{\partial{x_j}}(r_\delta)=\begin{cases}
1\cdot W_{\theta}(L_\delta,c_{\rho})-1\cdot0_A\ (j=n)\\
0\hspace{35.5mm}(j\neq n).
\end{cases}
\]
 By the definition of $f$-twisted Alexander matrices and previous discussions, we have
\[
A(Q(F),\rho;f_{\theta},0)=\left(
\begin{array}{ccccc}
-1\cdot 0_A & 0 & \cdots & 0 & 1\cdot a_1 \\
 0 & -1\cdot 0_A &  &  & 1\cdot a_2\\
 &  & \ddots &  & \vdots\\
 \vdots& & & -1\cdot 0_A & 1\cdot a_{n-1}\\
 0 &   &\cdots   & 0  & U
\end{array}
\right),
\]
where $U=\left( 1\cdot W_{\theta}(L_\delta,c_{\rho})-1\cdot 0_A\right)$. We remark that $U$ is a column vector with $k$ rows, where $k$ is the number of double point curves. We can see that $A(Q(F),\rho;f_{\theta},0)$ and $U$ are related by a finite sequence of the transformations (M1) $\sim$ (M8), which implies that $E_0(A(Q(F),\rho;f_{\theta},0))$ is the ideal generated by $\{1\cdot W_{\theta}(L_\delta,c_{\rho})-1\cdot0_A\mid \delta:\textrm{a double point curve}\}$. Thus, it holds that $E_0(A(Q(F),\rho;f_{\theta},0))\subset(\{1\cdot W_{\theta}(\lambda,c_{\rho})-1\cdot 0_A\mid \lambda\in H_1(F)\})$.

Let $\lambda$ be an element of $H_1(F)$. We take an oriented immersed curve $L_{\lambda}$ on $F$ based at ${\rm pr}^{-1}(q)$ which is a representative element of $\lambda$ and fix it. We may assume that $L^D_\lambda$ intersects the double point curves transversely, and misses triple points and branched points. We set $l_\lambda:=|L_\lambda\cap {\rm Ldeck}(D)|$. Let $e_{1}$ be the first point at which $L_\lambda$ intersects lower decker curves ${\rm Ldeck}(D)$ after departing from ${\rm pr}^{-1}(q)$ and let $e_{2},\ldots,e_{l_\lambda}$ be the points at which $L_\lambda$ intersects ${\rm Ldeck}(D)$ in turn with the orientation of $L_\lambda$. For $e_{i}$, we denote the upper sheet near the point $e_{i}^D$ by $v_{i}$. Then, the element $x_n^{v_{1}^{\varepsilon(e_{1})}\cdots v_{l_\lambda}^{\varepsilon(e_{l_\lambda})}}$ is equal to $x_n$ in the knot quandle $Q(F)$. Hence, we see that the presentation 
\[
\left\langle x_1,\ldots x_n~
\begin{array}{|c}
r_1,\ldots,r_{n-1}\\
r_\delta \ (\delta:\textrm{a double point curve})\\
(x_n^{v_{1}^{\varepsilon(e_{1})}\cdots v_{l_\lambda}^{\varepsilon(e_{l_\lambda})}},x_n)
\end{array}
\right\rangle
\]
is also a presentation of $Q(F)$. As in the previous discussion, we have 
\begin{eqnarray*}
\frac{\partial_{f_{\theta}\circ\rho^2}}{\partial{x_j}}\left(x_n^{v_{1}^{\varepsilon(e_{1})}\cdots v_{l_\lambda}^{\varepsilon(e_{l_\lambda})}}\right)&=&1\cdot\left(\sum^{l_\lambda}_{j=1}W_{\theta}(e_{j},c_{\rho})\right)\frac{\partial_{f_{\theta}\circ\rho^2}}{\partial{x_j}}(x_n)\\
&=&1\cdot W_{\theta}(L_\lambda,c_{\rho})\frac{\partial_{f_{\theta}\circ\rho^2}}{\partial{x_j}}(x_n)\\
&=&1\cdot W_{\theta}(\lambda,c_{\rho})\frac{\partial_{f_{\theta}\circ\rho^2}}{\partial{x_j}}(x_n).
\end{eqnarray*}
Thus, it holds that 
\[
\frac{\partial_{f_{\theta}\circ\rho^2}}{\partial{x_j}}\left(\left(x_n^{v_{i1}^{\varepsilon(e_{i1})}\cdots v_{il_i}^{\varepsilon(e_{il_i})}},x_n\right)\right)=\begin{cases}
1\cdot W_{\theta}(\lambda,c_{\rho})-1\cdot0_A\ (j=n)\\
0\hspace{36mm}(j\neq n).
\end{cases}
\]
This implies that $1\cdot W_{\theta}(\lambda,c_{\rho})-1\cdot 0_A$ is an element of the ideal $E_0(A(Q(F,\rho;f_{\theta},0)))$. Hence, the ideal $(\{1\cdot W_{\theta}(\lambda,c_{\rho})-1\cdot 0_A\mid \lambda\in H_1(F)\})$ is contained in $E_{0}(A(Q(F),\rho;f_{\theta},0))$. This implies the assertion.
\end{proof}
In the case of 2-knots, we have the following Corollary.
\begin{cor}
\label{cor:elementary_ideal_surface_knot}
Let $F$ be a 2-knot, $X$ a quandle, $A$ an abelian group and $\theta:X^2\to A$ a quandle 2-cocycle. For any quandle homomorphism $\rho:Q(F)\to X$, we have $E_0(A(Q(F),\rho;f_{\theta},0))=(0)$.
\end{cor}
\begin{proof}
By the definition of the Carter-Saito-Satoh's invariant, if $\lambda$ is the identity element of $H_1(F)$, we have $W_{\theta}(\lambda,c)=0_A$ for any $X$-coloring $c:S(D)\to X$ and quandle 2-cocycle $\theta:X^2\to A$. Thus, if $F$ is a 2-knot, it holds that
\[
E_0(A(Q(F),\rho;f_{\theta},0))=(\{1\cdot W_{\theta}(\lambda,c_{\rho})-1\cdot 0_A\mid \lambda\in H_1(F;\Z)\})=(0),
\]
where the first equality follows from Theorem \ref{theo:main_1}.
\end{proof}

\section{$f$-twisted Alexander matrices for connected quandles}
\label{sect:Alexander_matrix_connected_qdle}
In this section, we will discuss the $f$-twisted Alexander matrix of connected quandles with the Alexander pair obtained from a quandle $2$-cocycle. At first, we recall the definition of (co)homology groups of quandles \cite{Carter2003quandle}. Let $X$ be a quandle. For each positive integer $n$, let us denote the free abelian group whose basis is $X^n$ by $C^R_n(X)$. For $n\leq 0$, we assume $C^R_n(X)=0$. For each element $(x_1,\ldots,x_n)\in X^n$, we define an element $\partial(x_1,\ldots,x_n)$ of $C^R_{n-1}(X)$ by 
\begin{eqnarray*}
\partial(x_1,\ldots,x_n)&=&\sum^n_{i=2}(-1)^i(x_1,\ldots,x_{i-1},x_{i+1},\ldots,x_n)\\
&&-\sum^n_{i=2}(-1)^i(x_1^{x_i},\ldots,x_{i-1}^{x_i},x_{i+1},\ldots,x_n).
\end{eqnarray*}
Using this, for $n\ge 2$, we obtain a homomorphism $\partial_n:C^R_n(X)\to C^R_{n-1}(X)$. For $n\leq 1$, we define $\partial_n:C^R_n(X)\to C^R_{n-1}(X)$ by the zero map. Then, $(C^R_n(X),\partial_n)$ is a chain complex. 

Let $C^D_n(X)$ be the subgroup of $C^R_n(X)$ generated by the elements of
\[
\{ (x_1,\ldots,x_n)\in X^n\mid x_i=x_{i+1}\ \textrm{for some }i\}.
\]
We can verify that $(C^D_n(X),\partial_n)$ is a subcomplex of $(C^R_n(X),\partial_n)$. Thus, we obtain the chain complex $(C^Q_n(X)=C^R_n(X)/C^D_n(X),\partial_n)$. The homology group of the chain complex $(C^Q_n(X),\partial_n)$ is called the {\it quandle homology group}~\cite{Carter2003quandle}, which is denoted by $H^Q_n(X)$. 

Let $A$ be an abelian group. %We assume that the operation of $A$ is written by multiplicatively. 
For $n\in\Z$, let us define $C^n_Q(X;A)$ by the set of all group homomorphisms from $C^Q_n(X)$ to $A$ and $\delta_n:C^n_Q(X;A)\to C^{n+1}_Q(X;A)$ by $\delta_n(f)=f\circ\partial_{n+1}$. Then, we obtain a cochain complex $(C^n_Q(X),\delta_n)$, which is called the {\it quandle cochain complex}~\cite{Carter2003quandle}.

Let $\theta:X^2\to A$ be a quandle $2$-cocycle. We denote the linear extension of $\theta$ by the same symbol $\theta:\Z[X^2]\to A$. We see that the linear extension $\theta$ is a $2$-cocycle of the quandle cochain complex. Thus, given a quandle $2$-cocycle $\theta:X^2\to A$, we obtain the group homomorphism from $H^Q_2(X)$ to $A$. We also denote this group homomorphism by $\theta:H^Q_2(X)\to A$.

A quandle $X$ is {\it connected} if for every $x,y\in X$, it holds that $x^{z_1^{\varepsilon_1}\cdots z_n^{\varepsilon_n}}=y$ for some $z_1,\ldots,z_n\in X$ and $\varepsilon_1,\ldots,\varepsilon_n\in\{\pm 1\}$.  
\begin{theo}
\label{theo:main_2}
Let $Q=\langle x_1,\ldots,x_n\mid r_1,\ldots,r_m\rangle$ be a connected quandle, $X$ a quandle, $A$ an abelian group and $\theta:X^2\to A$ a quandle 2-cocycle. For any quandle homomorphism $\rho:Q\to X$, we have 
\[
E_0(A(Q,\rho;f_{\theta},0))=(\{1\cdot a-1\cdot 0_A\mid a\in{\rm Im}(\theta\circ\rho_{\ast})\}),
\]
 where $\rho_{\ast}:H^Q_2(Q)\to H^Q_2(X)$ is the group homomorphism induced by $\rho$.
\end{theo}
\begin{proof}
Let $\langle x_1,\ldots,x_n\mid r_1,\ldots, r_m\rangle$ be a presentation of $Q$. Since $Q$ is connected, for each $1\leq i\leq n-1$, there exist $y_{i1},\ldots,y_{il_i}\in Q$ and $\varepsilon_{i1},\ldots,\varepsilon_{il_i}\in \{\pm 1\}$ such that $x_i^{y_{i1}^{\varepsilon_{i1}}\cdots y_{il_i}^{\varepsilon_{il_i}}}=x_n$. Thus, using Tietze's moves, we may assume that 
\[
r_i=\begin{cases}
(x_i^{y_{i1}^{\varepsilon_{i1}}\cdots y_{il_i}^{\varepsilon_{il_i}}},x_n)\ (1\leq i\leq n-1)\\
(x_n^{y_{i1}^{\varepsilon_{i1}}\cdots y_{il_i}^{\varepsilon_{il_i}}},x_n)\ (n\leq i\leq m).
\end{cases}
\]
 We fix this presentation.

%Futhermore, replacing the relator $(x_i^{y_{i1}^{\varepsilon_{i1}}\cdots y_{im_i}^{\varepsilon_{im_i}}},x_n)$ to $(x_i^{x_i^{-\sum\varepsilon_{ik}}y_{i1}^{\varepsilon_{i1}}\cdots y_{im_i}^{\varepsilon_{im_i}}},x_n)$, we assume that $\sum^{l_i}_{k=1}\varepsilon_{ik}=0$ for $n\leq i\leq m$. This implies that $y_{i1}^{\varepsilon_{i1}}\cdots y_{im_i}^{\varepsilon_{im_i}}$ is contained in $[{\rm As}(X),{\rm As}(X)]$. 

%===============
\begin{comment}
\begin{eqnarray*}
\frac{\partial_{f_{\theta}}}{\partial x_j}(r_i)&=&\frac{\partial_{f_{\theta}}}{\partial x_j}(x_i\ast^{\varepsilon_{i1}}a_{i1}\cdots x_{i-1m_{i-1}}\ast^{\varepsilon_{il_i}}a_{il_i})-\frac{\partial_{f_{\theta}}}{\partial x_j}(x_n)\\
&=&\theta(x_i\ast^{\varepsilon_{i1}}a_{i1}\cdots x_{i-1m_{i-1}}\ast^{\frac{-1+\varepsilon_{il_i}}{2}}a_{il_i},a_{il_i})^{\varepsilon_{il_i}}\frac{\partial_{f_{\theta}}}{\partial x_j}(x_i\ast^{\varepsilon_{i1}}a_{i1}\cdots x_{i-1m_{i-1}})\\
&&-\frac{\partial_{f_{\theta}}}{\partial x_j}(x_n)
\end{eqnarray*}
\end{comment}
%===============
For any $1\leq j\leq m$, by Proposition \ref{prop:calculation}, it holds that 
\[
\frac{\partial_{f_{\theta}\circ\rho^2}}{\partial x_j}(x_i^{y_{i1}^{\varepsilon_{i1}}\cdots y_{il_i}^{\varepsilon_{il_i}}})=1\cdot\left(\sum^{l_i}_{k=1}\varepsilon_{ik}\theta(\rho(x_i^{y_{i1}^{\varepsilon_{i1}}\cdots y_{ik-1}^{\varepsilon_{ik-1}}y_{ik}^{\frac{\varepsilon_{ik}-1}{2}}}),\rho(y_{ik}))\right) \frac{\partial_{f_{\theta}\circ\rho^2}}{\partial x_j}(x_i)
\]
 if $1\leq i\leq n-1$, and 
\[
\frac{\partial_{f_{\theta}\circ\rho^2}}{\partial x_j}(x_n^{y_{i1}^{\varepsilon_{i1}}\cdots y_{il_i}^{\varepsilon_{il_i}}})=1\cdot\left(\sum^{l_i}_{k=1}\varepsilon_{ik}\theta(\rho(x_n^{y_{i1}^{\varepsilon_{i1}}\cdots y_{ik-1}^{\varepsilon_{ik-1}}y_{ik}^{\frac{\varepsilon_{ik}-1}{2}}}),\rho(y_{ik}))\right) \frac{\partial_{f_{\theta}\circ\rho^2}}{\partial x_j}(x_n)
\] 
if $n\leq i\leq m$. We put 
\[
a_i:=\begin{cases}
\sum^{l_i}_{k=1}\varepsilon_{ik}\theta(\rho(x_i^{y_{i1}^{\varepsilon_{i1}}\cdots y_{ik-1}^{\varepsilon_{ik-1}}y_{ik}^{\frac{\varepsilon_{ik}-1}{2}}}),\rho(y_{ik}))\ (1\leq i\leq n-1)\\
\sum^{l_i}_{k=1}\varepsilon_{ik}\theta(\rho(x_n^{y_{i1}^{\varepsilon_{i1}}\cdots y_{ik-1}^{\varepsilon_{ik-1}}y_{ik}^{\frac{\varepsilon_{ik}-1}{2}}}),\rho(y_{ik}))\ (n\leq i\leq m).
\end{cases}
\]
Then, we have 
\[
\frac{\partial_{f_{\theta}\circ\rho^2}}{\partial x_j}(r_i)=\begin{cases}
1\cdot a_i\frac{\partial_{f_{\theta}\circ\rho^2}}{\partial x_j}(x_i)-1\cdot 0_A\frac{\partial_{f_{\theta}\circ\rho^2}}{\partial x_j}(x_n)\ (1\leq i\leq n-1)\\
(1\cdot a_i-1\cdot 0_A)\frac{\partial_{f_{\theta}\circ\rho^2}}{\partial x_j}(x_n)\hspace{14mm} (n\leq i\leq m).
\end{cases}
\] %We note that $1\cdot z$ is a unit of $\Z[A]$ for any $z\in A$.

By the definition of the $f$-twisted Alexander matrix, it holds that 
\[
A(Q,\rho;f_{\theta},0)=\left(
\begin{array}{ccccc}
1\cdot a_1 & 0 & \cdots & 0 & -1\cdot 0_A \\
 0 & 1\cdot a_2 &  &  & -1\cdot 0_A\\
 &  & \ddots &  & \vdots\\
 \vdots& & & 1\cdot a_{n-1} & -1\cdot 0_A\\
  &   &   &   & 1\cdot a_n-1\cdot 0_A \\
  &   &   &   & \vdots \\
 0 &   & \cdots  & 0  & 1\cdot a_m-1\cdot 0_A \\
\end{array}
\right).
\]
As in the proof of Theorem \ref{theo:main_1}, we see that $A(Q,\rho;f_{\theta},0)$ and the matrix 
$\left(\begin{array}{c}
  1\cdot a_n-1\cdot 0_A  \\
  \vdots  \\
  1\cdot a_m-1\cdot 0_A
\end{array}\right)$ are related by a finite sequence of the transformations (M1) $\sim$ (M8). Hence, the 0-th elementary ideal $E_0(A(Q,\rho;f_{\theta},0))$ is the ideal generated by $1\cdot a_n-1\cdot 0_A,\ldots,1\cdot a_m-1\cdot 0_A$. 
%Thus, ${\rm Coker}(A(Q,\rho;f_{\theta},0))$ is isomorphic to the quotient module of $\Z[A]$ by the ideal generated by $1\cdot a_n-1\cdot 0_A,\ldots,1\cdot a_m-1\cdot 0_A$, which coincides with the $0$-th elementary ideal of $A(Q,\rho;f_{\theta},0)$. Hence, to prove Theorem \ref{theo:main_2}, we will show that $E_0(A(Q,\rho;f_{\theta},0))$ is equal to the ideal generated by $\{1\cdot a-1\cdot 0_A\mid a\in{\rm Im}(\theta\circ\rho_{\ast})\}$. At first, we will prove $E_0(A(Q,\rho;f_{\theta},0))\subset(\{1\cdot a-1\cdot 0_A\mid a\in{\rm Im}(\theta\circ\rho_{\ast})\})$.

 Next, we will show that the elements $a_n,\ldots,a_m$ are elements in ${\rm Im}(\theta\circ\rho_\ast)$. Let us consider the $2$-dimensional complex $\tilde{\Gamma}$ which is discussed in the section $8.3$ in \cite{Eisermann2014quandle} (see also \cite{Fenn1995trunks}). The complex $\tilde{\Gamma}$ is defined as follows: Let $\Gamma$ be the oriented graph with vertices $q\in Q$ and edges $(p\xrightarrow{q} r)$ for each triple $p,q,r\in Q$ with $p^q=r$. We regard $\Gamma$ as $1$-skeleton and glue in a 2-cell for each relation of $p^p=p,p^{qq^{-1}}=p^{q^{-1}q}=p$ and $p^{qr}=p^{rq^r}$. For convenience, the inverse path of $(p^{q^{-1}}\xrightarrow{q}p)$ is denoted by $(p\xrightarrow{q^{-1}}p^{q^{-1}})$. 

 By Theorem $9.9$ in \cite{Eisermann2014quandle}, a map $\varphi:C_1(\tilde{\Gamma})\to C^Q_2(Q)$ defined by mapping each edge $(p\xrightarrow{q}p^q )$ to $(p,q)$ induced an isomorphism $\varphi:H_1(\tilde{\Gamma})\to H^Q_2(Q)$. We note that $\varphi((p\xrightarrow{q^{-1}} p^{q^{-1}}))=-(p^{q^{-1}},q)$ for any $p,q\in Q$. We take a path $(p_0\xrightarrow{q_1^{\varepsilon_1}}\cdots\xrightarrow{q_l^{\varepsilon_l}}p_l)$ of $\Gamma$ and regard the path as a path in $\tilde{\Gamma}$. Then, it holds that 
\[
\varphi\left(\sum^l_{i=1}(p_{i-1}\xrightarrow{q_i^{\varepsilon_i}}p_i)\right)=\sum^{l}_{i=1}\varepsilon_i(p_{i-1}^{q_{i}^{\frac{\varepsilon_i-1}{2}}},q_{i}),
\]
where $p_i=p_0^{q_1^{\varepsilon_1}\cdots q_i^{\varepsilon_{i}}}$. 
%We remark that if $p_l=p_0$, $\sum^{l}_{i=1}\varepsilon_i(p_{i-1}^{q_{i}^{\frac{\varepsilon_i-1}{2}}},q_{i})$ is a $2$-cycle of $C^Q_2(Q)$.

 Assume that $p_0,q_1,\ldots,q_l$ are elements of $S=\{ x_1,\ldots,x_n\}$. 
By Proposition \ref{prop:calculation}, we have 
\[
\frac{\partial_{f_{\theta}\circ\rho^2}}{\partial x_j}(p_0^{q_1^{\varepsilon_1}\cdots q_l^{\varepsilon_l}})=1\cdot\left(\sum^{l}_{i=1}\varepsilon_i\theta(\rho(p_{i-1}^{q_{i}^{\frac{\varepsilon_i-1}{2}}}),\rho(q_{i}))\right)\frac{\partial_{f_{\theta}\circ\rho^2}}{\partial x_j}(p_0)
\]
 for any $1\leq j\leq n$. If $p_l=p_0$, the 2-chain $\sum^{l}_{i=1}\varepsilon_i(p_{i-1}^{q_{i}^{\frac{\varepsilon_i-1}{2}}},q_{i})$ is a $2$-cycle. This implies that the element $\sum^{l}_{i=1}\varepsilon_i\theta(\rho(p_{i-1}^{q_{i}^{\frac{\varepsilon_i-1}{2}}}),\rho(q_{i}))$ is an element of ${\rm Im}(\theta\circ\rho_{\ast})$. 

%By $x_n^{y_{i1}^{\varepsilon_{i1}}\cdots y_{il_i}^{\varepsilon_{il_i}}}=x_n$ in $Q$ for $n\leq i\leq m$, the $2$-chain $\sum^{l_i}_{k=1}\varepsilon_i(x_{n}^{y_{i1}^{\varepsilon_{i1}}\cdots y_{ik-1}^{\varepsilon_{k-1}}y_{ik}^{\frac{\varepsilon_{k}-1}{2}}},y_{k})$ is a $2$-cycle. 
By $x_n^{y_{i1}^{\varepsilon_{i1}}\cdots y_{il_i}^{\varepsilon_{il_i}}}=x_n$ in $Q$ for $n\leq i\leq m$, the element $a_i$ which satisfies $\frac{\partial_{f_{\theta}\circ\rho}}{\partial x_j}(x_n^{y_{i1}^{\varepsilon_{i1}}\cdots y_{il_i}^{\varepsilon_{il_i}}})=1\cdot a_{i}\frac{\partial_{f_{\theta}\circ\rho}}{\partial x_j}(x_n)$ is an element of the image of $\theta\circ\rho_{\ast}$ for any $n\leq i\leq m$. Thus, we see that $E_{0}(A(Q,\rho;f_{\theta},0))$ is contained in $(\{1\cdot a-1\cdot 0_A\mid a\in{\rm Im}(\theta\circ\rho_{\ast})\})$.
%Since $E_{0}(A(Q,\rho;f_{\theta},0))$ is the ideal generated by $1\cdot a_n-1\cdot0_A,\ldots,1\cdot a_m-1\cdot0_A$, we see that $E_{0}(A(Q,\rho;f_{\theta},0))$ is contained in $(\{1\cdot a-1\cdot 0_A\mid a\in{\rm Im}(\theta\circ\rho_{\ast})\})$.

For any $g\in H^Q_2(Q)$, there is a loop $(x_n\xrightarrow{y_1^{\varepsilon_1}}\cdots\xrightarrow{y_l^{\varepsilon_l}}x_n)$ in $\Gamma$ such that 
\[
\varphi\left(\sum^{l}_{i=1}(x_n^{y_1^{\varepsilon_1}\cdots y_{i-1}^{\varepsilon_{i-1}}}\xrightarrow{y_i^{\varepsilon_i}}x_n^{y_1^{\varepsilon_1}\cdots y_{i}^{\varepsilon_{i}}})\right)=g.
\]
 Using relation $p^{qr}=p^{rq^r}$, we may assume that $y_1,\ldots,y_l$ are elements in $S$.
% and $h_1,\ldots,h_l\in\{ x_1,\ldots,x_n\}$. 
By the above calculation, it holds that 
\[
\frac{\partial_{f_{\theta}\circ\rho^2}}{\partial x_j}(x_n^{y_1^{\varepsilon_1}\cdots y_l^{\varepsilon_l}})=1\cdot \theta\circ\rho_{\ast}(g)\frac{\partial_{f_{\theta}\circ\rho^2}}{\partial x_j}(x_n)
\]
for any $1\leq j\leq n$. It is obvious that the presentation 
\[
\langle x_1,\ldots,x_n\mid r_1,\ldots,r_m,(x_n^{y_1^{\varepsilon_1}\cdots y_l^{\varepsilon_l}},x_n)\rangle
\]
 is also a presentation of $Q$. Hence, the matrix $A(Q,\rho;f_{\theta},0)$ with respect to the presentation 
$\langle x_1,\ldots,x_n\mid r_1,\ldots,r_m,(x_n^{y_1^{\varepsilon_1}\cdots y_l^{\varepsilon_l}},x_n)\rangle$
 is related to the matrix 
$\left(
\begin{array}{c}
  1\cdot a_n-1\cdot 0_A  \\
  \vdots  \\
  1\cdot a_m-1\cdot 0_A \\
1\cdot \theta\circ\rho_{\ast}(g)-1\cdot 0_A
\end{array}
\right)$, which implies that $1\cdot \theta\circ\rho_{\ast}(g)-1\cdot 0_A$ is an element of $E_0(A(Q,\rho;f_{\theta},0))$. Thus, the ideal generated by $\{1\cdot a-1\cdot0_A\mid a\in{\rm Im}(\theta\circ\rho_{\ast})\}$ is contained in $E_{0}(A(Q,\rho;f_{\theta},0))$. 
%Since ${\rm Coker}(A(Q,\rho;f_{\theta},0))$ does not depend on a choice of presentations, it holds that $1\cdot \theta\circ\rho_{\ast}(g)-1\cdot e$ is contained in the ideal generated by $1\cdot u_n-1\cdot e,\ldots,1\cdot u_m-1\cdot e$. We obtain the claim.
\end{proof}

It is known that the knot quandle of a surface knot is connected. Hence, by Corollary \ref{cor:elementary_ideal_surface_knot} and Theorem \ref{theo:main_2}, we have the following Corollary.
\begin{cor}
\label{cor:2-knot_homology}
For any $2$-knot $F$, $H^Q_2(Q(F))$ is trivial.
\end{cor}
\begin{proof}
Let $F$ be a 2-knot. Assume that $H^Q_2(Q(F))$ is non-trivial. We set $A=H^Q_2(Q(F))$ and take a non-trivial element $a$ in $A$. Let $\theta:Q(F)^2\to A$ be a quandle 2-cocycle which corresponds to the identity map $H^Q_2(Q(F))\to A=H^Q_2(Q(F))$. 

Let $\rho$ be the identity map ${\rm id}:Q(F)\to Q(F)$. Since the composite map $\theta\circ\rho_{\ast}:H^Q_2(Q(F))\to A=H^Q_2(Q(F))$ is the identity map and the knot quandle $Q(F)$ is a connected quandle, $1\cdot a-1\cdot 0_A$ is contained in $E_0(A(Q(F),\rho;f_{\theta},0))$. This implies that $E_0(A(Q(F),\rho;f_{\theta},0))$ is not equal to zero ideal. On the other hand, by Corollary \ref{cor:elementary_ideal_surface_knot}, $E_0(A(Q(F),\rho;f_{\theta},0))$ is the zero ideal. This is a contradiction. 
\end{proof}
\begin{remark}
{\rm
By Artin's spinning construction \cite{Artin1925zur}, we see that that for any classical knot $K$, there exists a 2-knot $F$ such that $\pi_1(\R^3\backslash K)$ is isomorphic to $\pi_1(\R^4\backslash F)$. In other words, it holds that 
\[
\{\pi_1(\R^3\backslash K)\mid K:\textrm{ a classical knot}\}\subset\{\pi_1(\R^4\backslash F)\mid F:\textrm{ a 2-knot}\}.
\]

 It is known that for any non-trivial classical knot $K$, the second quandle homology group $H^Q_2(Q(K))$ is the infinite cyclic group \cite{Eisermann2003unknot}. Hence, by Corollary \ref{cor:2-knot_homology}, we see that for any non-trivial knot $K$, the knot quandle $Q(K)$ can not be realized by the knot quandle of 2-knots, that is, it holds that
\[
\{Q( K)\mid K:\textrm{ a non-trivial classical knot}\}\cap\{Q(F)\mid F:\textrm{ a 2-knot}\}=\emptyset.
\] 
}
\end{remark}

\bibliographystyle{plain}
\bibliography{reference}

\begin{bibdiv}
\begin{biblist}

\bib{Alexander1928topo}{article}{
      author={Alexander, J.~W.},
       title={Topological invariants of knots and links},
        date={1928},
        ISSN={0002-9947},
     journal={Trans. Amer. Math. Soc.},
      volume={30},
      number={2},
       pages={275\ndash 306},
  url={https://doi-org.remote.library.osaka-u.ac.jp:8443/10.2307/1989123},
}

\bib{Andruskiewitsch2003racks}{article}{
      author={Andruskiewitsch, Nicol{\'a}s},
      author={Gra\~{n}a, Mat{\i}as},
       title={From racks to pointed {H}opf algebras},
        date={2003},
     journal={Advances in Mathematics},
      volume={178},
      number={2},
       pages={177\ndash 243},
}

\bib{Artin1925zur}{article}{
      author={Artin, Emil},
       title={Zur {I}sotopie zweidimensionaler {F}l\"{a}chen im {$R_4$}},
        date={1925},
        ISSN={0025-5858},
     journal={Abh. Math. Sem. Univ. Hamburg},
      volume={4},
      number={1},
       pages={174\ndash 177},
         url={https://doi.org/10.1007/BF02950724},
}

\bib{Burde2014knots}{book}{
      author={Burde, Gerhard},
      author={Zieschang, Heiner},
      author={Heusener, Michael},
       title={Knots},
     edition={extended},
      series={De Gruyter Studies in Mathematics},
   publisher={De Gruyter, Berlin},
        date={2014},
      volume={5},
        ISBN={978-3-11-027074-7; 978-3-11-027078-5},
}

\bib{Carter2003quandle}{article}{
      author={Carter, J~Scott},
      author={Jelsovsky, Daniel},
      author={Kamada, Seiichi},
      author={Langford, Laurel},
      author={Saito, Masahico},
       title={Quandle cohomology and state-sum invariants of knotted curves and
  surfaces},
        date={2003},
     journal={Trans. Amer. Math. Soc.},
      volume={355},
      number={10},
       pages={3947\ndash 3989},
}

\bib{Carter2006ribbon}{article}{
      author={Carter, J~Scott},
      author={Saito, Masahico},
      author={Satoh, Shin},
       title={Ribbon concordance of surface-knots via quandle cocycle
  invariants},
        date={2006},
     journal={J. Aust. Math. Soc},
      volume={80},
      number={1},
       pages={131\ndash 147},
}

\bib{Carter2004surfaces}{book}{
      author={Carter, Scott},
      author={Kamada, Seiichi},
      author={Saito, Masahico},
       title={Surfaces in 4-space},
      series={Encyclopaedia of Mathematical Sciences},
   publisher={Springer-Verlag, Berlin},
        date={2004},
      volume={142},
        ISBN={3-540-21040-7},
         url={https://doi.org/10.1007/978-3-662-10162-9},
        note={Low-Dimensional Topology, III},
}

\bib{Eisermann2003unknot}{article}{
      author={Eisermann, Michael},
       title={Homological characterization of the unknot},
        date={2003},
        ISSN={0022-4049},
     journal={J. Pure Appl. Algebra},
      volume={177},
      number={2},
       pages={131 \ndash  157},
}

\bib{Eisermann2014quandle}{article}{
      author={Eisermann, Michael},
       title={Quandle coverings and their {G}alois correspondence},
        date={2014},
        ISSN={0016-2736},
     journal={Fund. Math.},
      volume={225},
      number={1},
       pages={103\ndash 168},
         url={https://doi.org/10.4064/fm225-1-7},
}

\bib{Fenn1992racks}{article}{
      author={Fenn, Roger},
      author={Rourke, Colin},
       title={Racks and links in codimension two},
        date={1992},
     journal={J. Knot Theory Ramifications},
      volume={1},
      number={04},
       pages={343\ndash 406},
}

\bib{Fenn1995trunks}{article}{
      author={Fenn, Roger},
      author={Rourke, Colin},
      author={Sanderson, Brian},
       title={Trunks and classifying spaces},
        date={1995},
        ISSN={0927-2852},
     journal={Appl. Categ. Structures},
      volume={3},
      number={4},
       pages={321\ndash 356},
         url={https://doi.org/10.1007/BF00872903},
}

\bib{Fox1953free}{article}{
      author={Fox, Ralph~H},
       title={Free differential calculus. {I}: Derivation in the free group
  ring},
        date={1953},
     journal={Ann. of Math.},
      volume={57},
      number={2},
       pages={547\ndash 560},
}

\bib{Ishii2022quandle}{article}{
      author={Ishii, Atsushi},
      author={Oshiro, Kanako},
       title={Quandle twisted {A}lexander invariants},
        date={2022},
        ISSN={0030-6126},
     journal={Osaka J. Math.},
      volume={59},
      number={3},
       pages={683\ndash 702},
}

\bib{Ishii2022twisted}{article}{
      author={Ishii, Atsushi},
      author={Oshiro, Kanako},
       title={Derivatives with {A}lexander pairs for quandles},
        date={2022},
        ISSN={0016-2736},
     journal={Fund. Math.},
      volume={259},
      number={1},
       pages={1\ndash 31},
         url={https://doi.org/10.4064/fm890-12-2021},
}

\bib{Joyce1982quandle}{article}{
      author={Joyce, David},
       title={A classifying invariant of knots, the knot quandle},
        date={1982},
     journal={J. Pure Appl. Algebra},
      volume={23},
      number={1},
       pages={37\ndash 65},
}

\bib{Kamada2017surface}{book}{
      author={Kamada, Seiichi},
       title={Surface-knots in 4-space. {A}n introduction},
      series={Springer Monographs in Mathematics},
   publisher={Springer, Singapore},
        date={2017},
}

\bib{Lin2001repr}{article}{
      author={Lin, Xiao~Song},
       title={Representations of knot groups and twisted {A}lexander
  polynomials},
        date={2001},
        ISSN={1000-9574},
     journal={Acta Math. Sin. (Engl. Ser.)},
      volume={17},
      number={3},
       pages={361\ndash 380},
  url={https://doi-org.remote.library.osaka-u.ac.jp:8443/10.1007/s101140100122},
}

\bib{Matveev1982distributive}{article}{
      author={Matveev, Sergei~Vladimirovich},
       title={Distributive groupoids in knot theory},
        date={1982},
     journal={Mat. Sb.},
      volume={161},
      number={1},
       pages={78\ndash 88},
}

\bib{Taniguchitwisted}{unpublished}{
      author={Taniguchi, Yuta},
       title={Alexander matrices of link quandles associated to quandle
  homomorphisms and quandle cocycle invariants},
        note={preprint, available at ar{X}iv:2107.06561},
}

\bib{Wada1994twisted}{article}{
      author={Wada, Masaaki},
       title={Twisted {A}lexander polynomial for finitely presentable groups},
        date={1994},
     journal={Topology},
      volume={33},
      number={2},
       pages={241\ndash 256},
}

\end{biblist}
\end{bibdiv}
\end{document}